\documentclass[a4paper,11pt]{article}
\usepackage[utf8]{inputenc}
\usepackage{amsthm}
\usepackage{amsmath}
\usepackage{amssymb}
\usepackage{caption}
\usepackage{cite}
\usepackage{subcaption}
\usepackage[labelformat=simple]{subcaption}
\usepackage{a4wide}
\usepackage{authblk}

\usepackage{silence}
\usepackage{tikz}
\usetikzlibrary{calc, fit, decorations.pathreplacing,decorations.pathmorphing}
\usepackage{tkz-graph} 
\tikzset{cycle2/.style={very thick, densely dashed}}

\usepackage{changes}

\usepackage{hyperref}

\usepackage{enumitem}
\setlist[enumerate]{noitemsep,topsep=3pt}
\setlist[itemize]{noitemsep,topsep=3pt}

\theoremstyle{definition}
\newtheorem{defin}{Definition}

\theoremstyle{plain}

\newcommand{\F}{\mathcal{F}}
\newcommand{\U}{\mathcal{U}}

\newcommand{\bfv}{\mathrm{bf}}
\newcommand{\bfvgraph}[1]{H_{#1}}
\newcommand{\h}{h}
\newcommand{\abs}[1]{\left\vert{#1}\right\vert}

\newtheorem{lemma}[defin]{Lemma}
\newtheorem{theorem}[defin]{Theorem}
\newtheorem{proposition}[defin]{Proposition}
\newtheorem{corollary}[defin]{Corollary}

\theoremstyle{definition}

\begin{document}

\title{On the Maximum Number of Vertices that Belong to Every Metric Basis \footnote{The paper has been presented in part in CALDAM 2025 \cite{Hakanen-conference-version}.}}

\date{November 11, 2026}

\author[1,2]{Anni Hakanen}
\author[1]{Ville Junnila}
\author[1]{Tero Laihonen}
\author[1]{Havu Miikonen\footnote{Corresponding author. Email: \texttt{havu.e.miikonen@utu.fi}.}}
\author[3]{Ismael G. Yero}
\affil[1]{Department of Mathematics and Statistics, University of Turku, Turku, Finland}
\affil[2]{Turku Collegium for Science, Medicine and Technology (TCSMT), University of Turku, Turku, Finland}
\affil[3]{Department of Mathematics, Universidad de C\'{a}diz  (Algeciras~Campus), Algeciras, Spain}

\maketitle

\begin{abstract}
Resolving sets of graphs have been widely studied since their introduction in the 1970's by Slater and, independently, by Harary and Melter. The minimum cardinality of a resolving set of a graph $G$ is called the metric dimension of $G$ and is denoted by $\dim(G)$. In this paper, we concentrate on the existence of vertices in a graph $G$ that belong to all metric bases of $G$, i.e. resolving sets of cardinality $\dim(G)$. Such vertices are called basis forced vertices, and we denote the number of them by $\bfv(G)$. We show that $\bfv(G)\le 2/3(n-\dim(G)-1)$ for any connected nontrivial graph $G$ of order $n$. In addition, we show that this bound can be attained. Furthermore, it also implies that $\bfv(G)\le 2/5(n-1)$ holds for any nontrivial connected graph $G$.
This result answers a question posed by Bagheri \emph{et al.} in 2016. Moreover, we provide a complete realization of the parameters $n$, $\dim(G)$ and $\bfv(G) \ge 1$ within the previous bounds. We consider some extremal cases related to basis forced vertices in a graph, in particular, we give a full characterization of the graphs with $\bfv(G) = 2$ and $\dim(G) = n-4$.


\end{abstract}

\section{Introduction}

A resolving set $R\subseteq V(G)$ in a graph $G=(V(G),E(G))$ represents a structure with the capability of uniquely recognizing all the vertices in $V(G)$ throughout a vector of distances to the vertices of $R$. The vertices of resolving sets are usually called ``landmarks'' and the optimization of the quantity of vertices in any resolving set has led to the notion of a metric basis, which is a resolving set of the smallest possible cardinality in $G$. These notions were first and separately introduced in \cite{Slater75} by Slater, and in \cite{Harary76} by Harary and Melter, but they remained almost without any attention until the work \cite{Chartrand}, which created a breaking point in the interest on them. Nowadays, the metric dimension of graphs is a very well-known topic and there is a huge amount of information about it. Some recent and interesting works on this parameter are for instance \cite{Claverol,Foster,Mashkaria,SedlarUnicyclic,Sedlar22,Sedlar21,Wang, Wu}. Moreover, for more information on this area, we suggest the very interesting survey \cite{till-2022+}.

Several applications of resolving sets have been proposed in literature. The connections of these structures to robot navigation already appeared in \cite{Harary76}, and have been further developed in several other works. A recent and novel location property was presented in \cite{tillquist-2019}, where (certain kind of) resolving sets were used while identifying biological sequence data. In order to avoid presenting a large list of such works, any interested reader might simply check the recent survey \cite{till-2022+}, which contains a fairly complete compendium of applications, and combinatorial or computational properties of resolving sets and metric bases.

An interesting fact, regarding the metric bases of a given graph $G$, relates to the possible existence of vertices of $G$ such that they are required to be landmarks in every metric basis, in order to locate or resolve the vertices of $G$. This fact means that a vertex satisfying this property plays a crucial role as a landmark, and thus, identifying such vertices is worthwhile. However, one cannot efficiently decide that a given vertex of a graph possesses such a property since, as proved in \cite{BasisForced}, it is a co-NP-hard problem to check whether a given vertex of a graph belongs to every metric basis of a graph. The vertices of a graph satisfying the previously mentioned property were called \textit{basis forced vertices} in \cite{BasisForced}. 

This suggests that finding tight bounds on the number of basis forced vertices in a given graph deserves attention. Hence, the aim of this work is to make significant contributions to this direction. Notice that in \cite{BagheriUnique16} Bagheri et al. studied graphs with unique metric bases, that is, graphs with a metric basis where all the vertices are basis forced. They provided bounds and posed a question on the maximum cardinality of a unique metric basis in any graph. In this paper, as a byproduct of Corollary~\ref{th:special-case-bound}, we are able to fully answer that open question. Some previous bounds for the number of basis forced vertices were already given in \cite{BasisForced}, and some other ones for specific families of graphs recently appeared in \cite{BasisForced-unic}, but the former ones (and more general), were indeed not the best possible. We significantly improve them in this work.  

Formally, given a connected graph $G=(V(G),E(G))$, it is said that a vertex $x\in V(G)$ {\em resolves} two vertices $u,v\in V(G)$ (or that $u,v$ are {\em resolved} by $x$), if $d_G(v,x)\ne d_G(u,x)$, where $d_G(y,z)$ represents the distance between $y$ and $z$, which is the number of edges in a shortest $y-z$ path in $G$. A set $R\subseteq V(G)$ is a \emph{resolving set} for $G$ if all the vertices of $G$ are pairwise resolved by a vertex of $R$. A resolving set having the smallest possible cardinality in $G$ is called a \emph{metric basis}. The cardinality of a metric basis is the \emph{metric dimension} of $G$, and denoted by $\dim(G)$.

Now, a \emph{basis forced vertex} of a graph $G$ is understood as a vertex $v\in V(G)$ such that it belongs to every metric basis of $G$. The term ``basis forced vertex'' was first used in \cite{BasisForced}, although it has some antecedents in the articles \cite{BagheriUnique16,BuczkowskiUnique},  where the graphs that have a unique metric basis were studied. From now on, we represent the number of basis forced vertices of $G$ by $\bfv(G)$. Notice that we can have $\bfv(G)=0$ for many graphs. For example, a nontrivial path $P$ has $\bfv(P)=0$. We also need the following concept from \cite{BasisForced}. A vertex $v\in V(G)$ is a \emph{void vertex} if it is in no metric basis of $G$. It is easy to see that there are $n-2$ void vertices in a path of order $n$; indeed, all the vertices not in the beginning or in the end of the path are such.

\subsection{Terminology and Notation}

Throughout our exposition, the graphs $G$, for which $\bfv(G)$ is studied, are nontrivial (that is, they have at least two vertices), finite, undirected and connected. The complete graph on $m$ vertices is denoted by $K_m$ and a path on $n$ vertices by $P_n$. The complement of a graph $G$ is denoted by $\overline{G}$. The cardinality of a set $X$ is denoted by $|X|$. The notation $n(G)$ means the order $|V(G)|$ of a graph $G$.
We use $uv$ as a shorthand notation for the edge $\{u,v\}$.

Given a graph $G$ and a set $R \subseteq V(G)$, in~\cite{BasisForced}, the \textit{colour graph} $G_R$ of $G$ (with respect to $R$) is defined as follows. Let $r \in R$. We denote
\[
\U_R (r) = \{ \{x,y\} \in V(G)^2 \ | \ d_G(r,x) \neq d_G(r,y) \text{ and } \forall \, t \in R \setminus \{r\} \colon d_G(t,x) = d_G(t,y) \}.
\]
In other words, the set $\U_R (r)$ consists of the pairs of vertices for which $r$ is the unique element in $R$ that resolves the pairs.

With this in mind, the graph $G_R$ has the vertex set $V(G_R)=V(G)$ and the edge set
\[
\bigcup\limits_{r\in R} \U_R (r).
\]
Each $r \in R$ is assigned a colour (or a label), and we colour the edges in $G_R$ given by $\U_R (r)$ with the colour associated with $r$. Observe that the graph $G_R$ can be disconnected while $G$ is connected. Moreover, note that this graph $G_R$ might be constructed for any set of vertices $R\subseteq V(G)$, although for our purposes, we will require that such $R$ is a resolving set (or a metric basis). If there is no edge between $x$ and $y$ in $G_R$, then there are at least two elements in a resolving set $R$ that resolve $x$ and $y$.  If $R$ is a metric basis of $G$, then the graph $G_R$ has at least one edge of the colour associated with each $r \in R$. In other words, the set $\U_R (r)$ is nonempty for all $r\in R$.

For an example (see~\cite{BasisForced}), consider the graph $G$ illustrated in Figure \ref{fig:colourExG} and its resolving set $R=\{r_1,r_2\}$.
For the colour graph $G_R$, we first form the following sets:
$\U_R (r_1) = \{\{r_1,v_1\}, \{r_1,v_3\},\{v_1,v_3\}\},\{v_2,v_4\}\}$ and $\U_R (r_2) = \{\{r_2,v_1\}, \{r_2,v_2\},\{v_1,v_2\}\},\{v_3,v_4\}\}$. Then we obtain the colour graph $G_R = (V(G_R), E(G_R))$, where $V(G_R) = V(G)$ and $E(G_R) = \U_R (r_1) \cup \U_R (r_2)$. In Figure \ref{fig:colourExGR} illustrating the colour graph $G_R$, the edges corresponding to $r_1$ and $r_2$ are associated with black and ``dashed'' edges, respectively.

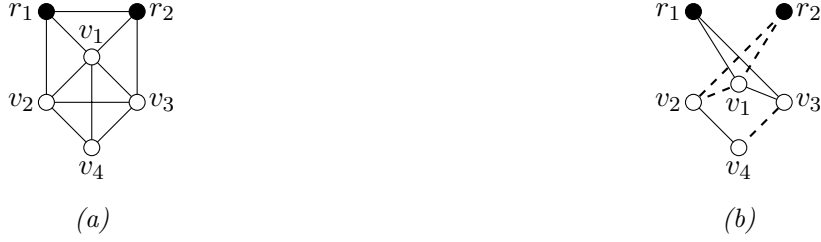
\begin{figure}[t]
\centering
\begin{subfigure}[b]{0.45\linewidth}
\centering
\begin{tikzpicture}[scale=.6]
	\draw (0,2) -- (1,3) -- (1,1) -- (0,2) -- (-1,3) -- (-1,1) -- (0,2) -- (0,0) -- (1,1) -- (-1,1) -- (0,0) (-1,3) -- (1,3);
 	\draw \foreach \x in {(0,0),(-1,1),(1,1),(0,2)} {
 		\x node[circle, draw, fill=white, inner sep=0pt, minimum width=6pt] {}
 	};
 	\draw \foreach \x in {(-1,3),(1,3)} {
 		\x node[circle, draw, fill=black, inner sep=0pt, minimum width=6pt] {}
 	};
	\draw {
		(-1.55,3) node[] {$r_1$}
		(1.55,3) node[] {$r_2$}
		(0,-.5) node[] {$v_4$}
		(-1.55,1) node[] {$v_2$}
		(1.55,1) node[] {$v_3$}
		(0,2.5) node[] {$v_1$}
	};
 \end{tikzpicture}
 \caption{ }\label{fig:colourExG}
\end{subfigure}
\hfill
\begin{subfigure}[b]{0.45\linewidth}
\centering
\begin{tikzpicture}[scale=.6]
	\draw (-1,3) -- (1,1) -- (0,1.4) -- (-1,3) (0,0) -- (-1,1);
	\draw[dashed,thick] (1,3) -- (-1,1) -- (0,1.4) -- (1,3) (0,0) -- (1,1);
 	\draw \foreach \x in {(0,0),(-1,1),(1,1),(0,1.4)} {
 		\x node[circle, draw, fill=white, inner sep=0pt, minimum width=6pt] {}
 	};
 	\draw \foreach \x in {(-1,3),(1,3)} {
 		\x node[circle, draw, fill=black, inner sep=0pt, minimum width=6pt] {}
 	};
	\draw {
		(-1.55,3) node[] {$r_1$}
		(1.55,3) node[] {$r_2$}
		(0,-.5) node[] {$v_4$}
		(-1.55,1) node[] {$v_2$}
		(1.55,1) node[] {$v_3$}
		(0,.9) node[] {$v_1$}
	};
 \end{tikzpicture}
 \caption{ }\label{fig:colourExGR}
\end{subfigure}
\caption{(a) The graph $G$ with the resolving set $R = \{r_1,r_2\}$ and (b) the colour graph $G_R$.}
\label{fig:colourEx}
\end{figure}

\section{Properties of the Colour Graph and Bounds for the Metric Dimension}

The colour graph $G_R$ regarding a resolving set $R$ in a graph $G$ plays an important role in the proofs of this paper. Therefore, we begin by giving some properties of the graph $G_R$ stated in the  next technical lemmas. The results of the Lemma~\ref{lem:colorprops} are already known from \cite{BasisForced}, whereas the claims in Lemma~\ref{lem:colorprops-more} are new.
In this section, we also provide (in Proposition~\ref{prop:n-4}) bounds on the metric dimension $\dim(G)$ when there are basis forced vertices in $G$. We will need these bounds for our main results in Sections~\ref{sec:main} and~\ref{sec:realization}.

\begin{lemma}[\cite{BasisForced}]
\label{lem:colorprops}
Let $G$ be a graph and let $R$ be a resolving set of $G$. Then the following properties hold for $G_R$.
\begin{enumerate}[label=\textnormal{(\roman*)}]
\item A colour that appears in a cycle of $G_R$ appears at least twice in that cycle.

\item Let $x,y,z \in V(G)$. If the edges $xy$ and $xz$ have the same colour in $G_R$, then the edge $yz$ also has the same colour in $G_R$.

\item If $b \in V(G)$ is a basis forced vertex of $G$ and $R$ is a metric basis of $G$, then the graph $G_R$ has at least two edges of the colour associated with $b$.

\item If $b \in V(G)$ is a basis forced vertex of $G$ and $R$ is a metric basis of $G$, then the graph $G_R$ has at least one edge $xy$, where $x,y \in V(G) \setminus R$, of the colour associated with $b$.

\item The set of vertices $R$ forms an independent set in $G_R$.

\item If there is an edge in $G_R$ incident to $r\in R$, then the edge has the colour associated with~$r$.
\end{enumerate}
\end{lemma}

\begin{proof} The items (i)--(iv) are from Lemma~25 in \cite{BasisForced} and the items (v)--(vi) are basic observations of the colour graph discussed at the beginning of Section 4 in \cite{BasisForced}.
\end{proof}

We next show some new properties of the colour graph $G_R$ of a graph $G=(V(G),E(G))$ with respect to some resolving set or metric basis $R$. If we replace a vertex $u$ in $R$ by a vertex $v\in V(G)$ (while keeping $R$ otherwise intact), then we denote the new set by $R[u\leftarrow v]$. In other words,  $$R[u\leftarrow v]=(R\setminus\{u\})\cup \{v\}.$$

\begin{lemma}
\label{lem:colorprops-more}
Let $G$ be a graph and let $R$ be a resolving set of $G$. Then the following properties hold for $G_R$.
\begin{enumerate}[label=\textnormal{(\roman*)}]
    \item If $b \in V(G)$ is a basis forced vertex of $G$ and $R$ is a metric basis of $G$, then the graph $G_R$ has at least two edges $xy$, where $x,y \in V(G) \setminus R$, of the colour associated with $b$.
    
    \item If $R$ is a metric basis of $G$ and the only edge of the colour associated with $r \in R$ in $G_R$  is of type $rx$, where $x \in V(G) \setminus R$, then $R[r\leftarrow x]$ is a metric basis of $G$.
\end{enumerate}
\end{lemma}

\begin{proof}
(i) Let us assume that $b\in V(G)$ is a basis forced vertex. By Lemma~\ref{lem:colorprops}(iv) we know that there is at least one edge, say $xy$ such that $\{x, y\} \subseteq V(G)\setminus R$ associated with the colour $b$ and, by Lemma~\ref{lem:colorprops}(iii), there is also at least one more edge, say $wz$, of the same colour. If $\{w,z\}\subseteq V(G)\setminus R$, we are done. Clearly, by Lemma~\ref{lem:colorprops}(v),  we cannot have $\{w,z\}\subseteq R$, so it suffices to assume that $w\in R$. Moreover, $w=b$ by Lemma~\ref{lem:colorprops}(vi). Next we consider separately two possible cases, namely,  $z\in\{x,y\}$ or $z\notin \{x,y\}.$ 

\smallskip
\noindent
\emph{Case 1}: $z\notin \{x,y\}.$  If there is a third edge of colour $b$, then it is enough to assume that it is $bz'$ where $z'\in V(G)\setminus R$. However, in that case, by Lemma~1(ii) there exists an edge $zz'$ with $z, z' \in V(G) \setminus R$ and the claim follows. Hence, we may assume that $xy$ and $bz = wz$ are the only edges of colour $b$. The set $R[b\leftarrow x]$ cannot be a metric basis, since $b$ is a basis forced vertex which must be in every metric basis. Therefore, $d_G(x,z)=d_G(x,b)$ as the pair $z$ and $b$ is the only one that cannot be resolved with respect to $R[b\leftarrow x]$ (indeed, all the other pairs of vertices in $V(G)$ apart from the pair $b$ and $z$ and the pair $x$ and $y$ were resolved by other elements of $R$ than $b$). Similarly, as $R[b\leftarrow y]$ (resp. $R[b\leftarrow z]$) cannot be a metric basis, we have $d_G(y,z)=d_G(y,b)$ (resp. $d_G(z,x)=d_G(z,y)$). These three distance equations imply together that $d_G(b,y)=d_G(z,y)=d_G(z,x)=d_G(b,x).$ However, this is a contradiction, since $b$ was the (only) vertex resolving the vertices $y$ and $x$, that is $d_G(b,y)\neq d_G(b,x)$.

\smallskip
\noindent
\emph{Case 2}: $z\in\{x,y\}$. Without loss of generality, assume $z=x$. Recall that by Lemma~\ref{lem:colorprops}(ii) there is a third edge associated with the colour $b$, namely, $by$ in $G_R$. If there is yet another (fourth) edge of colour $b$, then we can assume that it is $bz'$ where $z'\in V(G)\setminus R$. Due to Lemma~\ref{lem:colorprops}(ii), this implies that there is  an edge $z'x$ with $x, z' \in V(G) \setminus R$ of colour $b$, and we are done. Hence we may assume that we have only the three edges mentioned above. Because $b$ is a basis forced vertex, the set $R[b\leftarrow y]$ is not a metric basis and, hence, $d_G(y,b)=d_G(y,x)$. Analogously, $R[b\leftarrow x]$ cannot be a metric basis.  Hence, we obtain $d_G(x,b)=d_G(x,y)$. Now we have $d_G(b,x)=d_G(x,y)=d_G(b,y)$, a contradiction, since $b$ resolves the pair $x$ and $y$.

\medskip

(ii) Let $R$ be a metric basis (when $R$ in only a resolving set, the claim follows by an analogous argument). Assume further that $r\in R$, and $rx$ with $x\notin R$ is the only edge of colour associated with $r$ (clearly, $r$ is not a basis forced vertex). If we remove $r$ from $R$, then the only pair of vertices that can have the same distances to all the elements of $R\setminus\{r\}$ is $r$ and $x$. But with respect to the set $R[r\leftarrow x]$, they have different distances to $x$. Therefore, the set $R[r\leftarrow x]$ is a metric basis in $G$.
\end{proof}

The next result shows that if we have basis forced vertices in a graph $G$, then there are some limitations for the metric dimension $\dim(G)$.

\begin{proposition}
\label{prop:n-4}
If $G$ is a graph such that $\bfv(G) \geq 1$, then $2 \leq \dim(G) \leq n-4$.
\end{proposition}
\begin{proof}
Assume that $G$ is a graph with at least one basis forced vertex. Recall that paths are the only graphs with $\dim(G) = 1$. Hence, as the paths do not have basis forced vertices, it immediately follows that $\dim(G) \geq 2$. Let us look at the claim $\dim(G)\le n-4$. Suppose to the contrary that  $\dim(G)\ge n-3$ (and $\bfv(G) \geq 1$). Let $R$ be a metric basis of $G$. By Lemma~\ref{lem:colorprops-more}(i), for each basis forced vertex $u$, there exist at least two edges associated with the colour $u$ in $G_R$ between vertices of $V (G) \setminus R$. Hence, $\dim(G) \leq n-3$ as otherwise at most one edge can occur in $V (G) \setminus R$. Thus, $\dim(G) = n-3$ and there are exactly three vertices, say $x$, $y$ and $z$, in $V (G)\setminus R$. In addition, we know that there  exists a unique basis forced vertex in $G$, say $b$. Indeed, if $\bfv(G)>1$, then Lemma~\ref{lem:colorprops-more}(i) would imply that $\dim(G)\le n-4$.

By Lemma~\ref{lem:colorprops-more}(i), there exist at least two edges associated with the colour $b$ in $G_R$ between the vertices of $V(G) \setminus R$. Without loss of generality, we may assume that $xy$ and $yz$ are such edges. By Lemma~\ref{lem:colorprops}(ii), this further implies that $xz$ is also such an edge. Let $r$ be a vertex in $R \setminus \{b\}$. Since $R$ is a metric basis of $G$, there exists at least one edge in $G_R$ associated with the colour $r$. Due to the fact that the edges between $x$, $y$ and $z$ are associated with the colour $b$, the edges of the colour $r$ have to be of type $rw$, where $w \in \{x,y,z\}$ due to claims (vi) and (v) in Lemma~\ref{lem:colorprops}. Moreover, there exists exactly one such edge since otherwise a contradiction (with the fact that the edges between $x$, $y$ and $z$ are coloured with $b$) follows by Lemma~\ref{lem:colorprops}(ii). For each $w \in \{x,y,z\}$, we define
\[
R_w = \{w\} \cup \{u \in R \setminus \{b\} \mid uw \in E(G_R)\}\text.
\]
It is immediate that the sets $\{b\}$, $R_x$, $R_y$ and $R_z$ form a partition of $V(G)$. Furthermore, we have the following observations on $R_w$:
\begin{itemize}
	\item[(1)] By Lemma~\ref{lem:colorprops-more}(ii), the set $R[w \leftarrow v]$ is a metric basis for each $v \in R_w \setminus \{w\}$.
	\item[(2)] For each $r \in R_w \setminus \{w\}$, we have $d_G(b,w) = d_G(b,r)$ as the vertices $r$ and $w$ are solely resolved by $r$. Thus, the basis forced vertex $b$ has the same distance to all $u \in R_w$.
\end{itemize}

Since the vertices $x$, $y$ and $z$ are resolved from each other by the vertex $b$, we may without loss of generality assume that $d_G(b,x) < d_G(b,y) < d_G(b,z)$. Thus, based on the partition $\{b\}$, $R_x$, $R_y$ and $R_z$, we have that $d_G(b,x) = 1$, $d_G(b,y) = 2$ and $d_G(b,z) = 3$. Hence, by the observation~(2) above, it happens $d_G(b,r_x) = 1$, $d_G(b,r_y) = 2$ and $d_G(b,r_z) = 3$ for all $r_x \in R_x \setminus \{x\}$, $r_y \in R_y \setminus \{y\}$ and $r_z \in R_z \setminus \{z\}$.

Let $bx'y'z$ be a shortest path of length $3$ from $b$ to $z$. It is immediate that $x' \in R_x$ and $y' \in R_y$; note that $x'$ and $y'$ can be $x$ and $y$, respectively. By the observation~(1), $R' = R[x' \leftarrow x]$ is a resolving set of $G$. As above, we may deduce that the edges $x'y$, $x'z$ and $yz$ have colour $b$ in $G_{R'}$. Moreover, in $G_{R'}$ there exists exactly one edge of colour $y'$, namely, $yy'$. Therefore, by the observation~(1), $R'' = R'[y' \leftarrow y]$ is a metric basis of $G$. Now $V(G) \setminus R'' = \{x',y',z\}$. Furthermore, $R''[b \leftarrow z]$ is a metric basis of $G$ since $d_G(z,y') = 1$, $d_G(z,x') = 2$ and $d_G(z,b) = 3$. However, this contradicts the fact that $b$ is a basis forced vertex. Thus, the claim follows.
\end{proof}

Notice that there exist graphs attaining the upper bound of the proposition. For example, the graph $G$ of order $6$ in Figure~\ref{fig:colourEx} satisfies $\bfv(G) = 2 (> 1)$ and $\dim(G) = 2 = n-4$. In Section~\ref{CharSec}, we give a full characterization of the graphs with $\bfv(G) = 2$ and $\dim(G) =  n-4$.

\section{The Main Bounds for the Maximum Number of Basis Forced Vertices} 
\label{sec:main}

The main result of this section presents a bound for the number of basis forced vertices of a graph $G$ in terms of the metric dimension $\dim(G)$ and the order $n$ of $G$ (see Theorem~\ref{th:main-bound}). As a special case, another bound for such quantity is given only in terms of the order of $G$ (see Corollary~\ref{th:special-case-bound}). 

\smallskip

A \emph{cactus} is a graph in which no two cycles share an edge. In what follows, we want to show that a cactus graph appears as a certain subgraph of the colour graph $G_R$. To make the proof of Lemma~\ref{lem:H-properties} simpler, we present the following lemma, the proof of which is based on an iterative idea with an algorithmic flavour.
The lemma shows that given a graph with two cycles sharing an edge, we can obtain a pair of cycles such that their intersection is a path (forming a so called theta graph).
Recall that we define \emph{cycle} as a closed walk with no repeated vertices.
In the next lemma, we use the term \emph{simple cycle} to emphasise the fact that cycles have no repeated vertices. Moreover, we ``abuse'' the notation and use the intersection notation $C_1 \cap C_2$ of two cycles (as subgraphs) to represent the subgraph induced by the vertices of the intersection $V(C_1) \cap V(C_2)$.

\begin{lemma}
\label{lem:cactus-theta}
Let $G$ be a graph and let $C_1$ and $C_2$ be simple cycles in $G$.
If $C_1 \cap C_2$ contains at least one edge, then there exist a third cycle $C_3$ in $G$ such that $C_1 \cap C_3$ is a nontrivial path $($a path of length at least $2$$)$. 
\end{lemma}

\begin{proof}
If the intersection $C_1 \cap C_2$ is a nontrivial path, then we are done.
Otherwise $C_1 \cap C_2$ contains two or more connected components. 
We denote the vertices of $C_1$ by $v_1, v_2, \dots, v_{n_1}$ starting from the beginning of a section that $C_1$ shares with $C_2$, or more formally, $v_1 v_2 \in E(C_2)$ and $v_{n_1}v_1 \notin E(C_2)$.
Such edges exist because $C_1$ and $C_2$ share at least one edge and they are distinct cycles.
We use the notation $x \xrightarrow{C_1} y$ to denote the path from $x$ to $y$ in which the vertices are in ascending order with respect to their indices. 
The notation $x \xrightarrow{C_1} x$ is just the vertex $x$.
The concatenation of paths $x \xrightarrow{C_1} y$ and $y \xrightarrow{C_1} z$ is denoted naturally by $x \xrightarrow{C_1} y \xrightarrow{C_1} z$.
    
Let $i$ be the largest index such that the path $v_1 \xrightarrow{C_1} v_i$ is contained in $C_2$.
Let $j>i$ be the smallest index and $k \ge j$ the largest index such that the path $v_j \xrightarrow{C_1} v_k$ is contained in $C_2$.
In other words, the paths $v_1 \xrightarrow{C_1} v_i$ and $v_j \xrightarrow{C_1} v_k$ are the first two connected components that the cycles $C_1$ and $C_2$ have in common with respect to the numbering of the vertices of $C_1$.
Notice that it is possible that $j = k$.
The vertices $v_{i+1}, \dots, v_{j-1}$ (if there are any) are not in the cycle $C_2$.
    
We number the vertices of $C_2 = u_1u_2 \cdots u_{n_2}u_1$ so that the numbering satisfies $u_1 = v_1$, $u_2 = v_2$, ...,  $u_i = v_i$.
The rest of the vertices along $C_2$ are then numbered in order. 
Then we define the notation $x \xrightarrow{C_2} y$ analogously to $x \xrightarrow{C_1} y$.

The choice of a new cycle $C'_2$ depends on how the shared paths $v_1\cdots v_i$ and $v_j\cdots v_k$ are arranged in $C_2$.

\smallskip
\noindent
\emph{Case 1}:
If the vertex $v_j$ appears before $v_k$ in the numbering of $C_2$, or if $v_j = v_k$ (see Figure~\ref{fig:cycles}(a)), we choose
$$C_2' = v_1 \xrightarrow{C_2} v_i \xrightarrow{C_1} v_j \xrightarrow{C_2} v_k \xrightarrow{C_2} v_1.$$

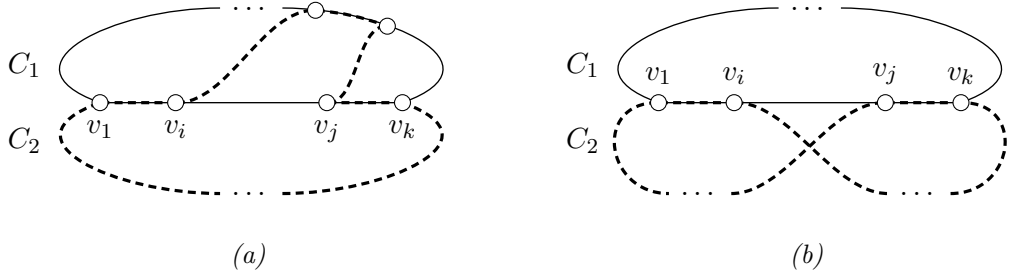
\begin{figure}[t]
    \centering
    \begin{subfigure}[t]{0.45\textwidth}
        \centering
        \begin{tikzpicture}[yscale=1, xscale=1, rotate=0]

    \renewcommand*{\EdgeLineWidth}{ 0.5pt}
    \SetVertexMath
    \SetVertexLabelOut

    {
    \SetVertexNoLabel
    \tikzstyle{VertexStyle}=[minimum size=1pt,inner sep=0pt]

    \Vertex[x=0, y=0]{1}
    \Vertex[x=1, y=0]{2}
    \Vertex[x=3, y=0]{3}
    \Vertex[x=4, y=0]{4}

    \Vertex[x=2.85, y=1.22]{5}
    \Vertex[x=3.8, y=1.01]{6}

    }

    {
    \tikzstyle{VertexStyle}=[fill=white,circle]
    \SetVertexLabelIn
    \Vertex[x=2, y=1.25, L=\dots]{dots1}
    \Vertex[x=2, y=-1.2, L=\dots]{dots2}
    \Vertex[x=-1, y=0.5, L=C_1]{C1}
    \Vertex[x=-1, y=-.5, L=C_2]{C2}
    }

    \tikzstyle{EdgeStyle}=[]
    \Edge(1)(2)
    \Edge(2)(3)
    \Edge(3)(4)

    \tikzstyle{EdgeStyle}=[out = 160, in = 180, looseness=1.8]
    \Edge(1)(dots1)
    
    \tikzstyle{EdgeStyle}=[out = 20, in = 0, looseness=1.8]
    \Edge(4)(dots1)

    \tikzstyle{EdgeStyle}=[out = 0, in = 180, looseness=0.8, cycle2]
    \Edge(2)(5)
    \Edge(3)(6)
    \Edge(1)(2)
    \Edge(3)(4)

    \tikzstyle{EdgeStyle}=[bend left = 7, cycle2]
    \Edge(5)(6)
    
    \tikzstyle{EdgeStyle}=[out = 0, in = -20, looseness=1.8, cycle2]
    \Edge(dots2)(4)
    
    \tikzstyle{EdgeStyle}=[out = 200, in = 180, looseness=1.8, cycle2]
    \Edge(1)(dots2)

    \tikzstyle{VertexStyle}=[circle,draw, fill=white,inner sep=0pt, minimum width=6pt]
    \SO[unit=0, Lpos=-90](1){v_1}
    \SO[unit=0, Lpos=-90](2){v_i}
    \SO[unit=0, Lpos=-90](3){v_j}
    \SO[unit=0, Lpos=-90](4){v_k}

    {
    \SetVertexNoLabel
    \SO[unit=0](5){aaa}
    \SO[unit=0](6){bb}
    }

\end{tikzpicture}

        \caption{}
     \end{subfigure}
     ~
     \begin{subfigure}[t]{0.45\textwidth}
        \centering
        \begin{tikzpicture}[yscale=1, xscale=1, rotate=0]

    \renewcommand*{\EdgeLineWidth}{ 0.5pt}
    \SetVertexMath
    \SetVertexLabelOut

    {
    \SetVertexNoLabel
    \tikzstyle{VertexStyle}=[minimum size=1pt,inner sep=0pt]

    \Vertex[x=0, y=0]{1}
    \Vertex[x=1, y=0]{2}
    \Vertex[x=3, y=0]{3}
    \Vertex[x=4, y=0]{4}

    }

    {
    \tikzstyle{VertexStyle}=[fill=white,circle]
    \SetVertexLabelIn
    \Vertex[x=2, y=1.25, L=\dots]{dots1}
    \Vertex[x=0.55, y=-1.2, L=\dots]{dots2}
    \Vertex[x=3.45, y=-1.2, L=\dots]{dots3}
    \Vertex[x=-1, y=0.5, L=C_1]{C1}
    \Vertex[x=-1, y=-.5, L=C_2]{C2}
    }

    \tikzstyle{EdgeStyle}=[]
    \Edge(1)(2)
    \Edge(2)(3)
    \Edge(3)(4)

    \tikzstyle{EdgeStyle}=[out = 160, in = 180, looseness=1.8]
    \Edge(1)(dots1)
    
    \tikzstyle{EdgeStyle}=[out = 20, in = 0, looseness=1.8]
    \Edge(4)(dots1)

    \tikzstyle{EdgeStyle}=[out = 0, in = 180, looseness=0.8, cycle2]
    \Edge(2)(dots3)
    \Edge(dots2)(3)
    \Edge(1)(2)
    \Edge(3)(4)
    
    \tikzstyle{EdgeStyle}=[out = 0, in = 0, looseness=1.8, cycle2]
    \Edge(dots3)(4)
    
    \tikzstyle{EdgeStyle}=[out = 180, in = 180, looseness=1.8, cycle2]
    \Edge(1)(dots2)

    \tikzstyle{VertexStyle}=[circle,draw, fill=white, inner sep=0pt, minimum width=6pt]
    \SO[unit=0, Lpos=90](1){v_1}
    \SO[unit=0, Lpos=90](2){v_i}
    \SO[unit=0, Lpos=90](3){v_j}
    \SO[unit=0, Lpos=90](4){v_k}

\end{tikzpicture}

        \caption{}
     \end{subfigure}
    \caption{The cycles $C_1$ and $C_2$ can intersect in essentially two different ways. The cycle $C_2$ is drawn with a dashed line. The lines represent nontrivial paths $($not necessarily edges$)$, with the exception of the line from $v_j$ to $v_k$ if $v_j = v_k$.}
    \label{fig:cycles}
\end{figure}

\smallskip
\noindent
\emph{Case 2}:
If the vertex $v_j$ appears after $v_k$ in the numbering of $C_2$ (see Figure~\ref{fig:cycles}(b)), we choose
$$C_2' = v_1 \xrightarrow{C_2} v_i \xrightarrow{C_1} v_j \xrightarrow{C_2} v_1.$$

Let us observe that the closed walk $C_2'$ is a simple cycle in both cases.
Indeed, there are no repeated vertices in the sections of the form $* \xrightarrow{C_2} *$ (apart from $v_1$) because $C_2$ is a simple cycle.
By the choice of $v_i$ and $v_j$, the path $v_i \xrightarrow{C_1} v_j$ does not intersect $C_2$ in vertices other than its endpoints.
    
In both cases, the intersection $C_1 \cap C_2'$ has at least one fewer connected component than $C_1 \cap C_2$ does.
The disconnected paths $v_1 \xrightarrow{C_1} v_i$ and $v_j \xrightarrow{C_1} v_k$ are replaced by one (nontrivial) path, $v_1 \xrightarrow{C_1} v_k$ (in Case 1) or $v_1 \xrightarrow{C_1} v_j$ (in Case 2). 
In addition, the path $v_i \xrightarrow{C_2} v_j$ (or $v_i \xrightarrow{C_2} v_k$) may intersect with $C_1$ in vertices other than the endpoints (see Figure~\ref{fig:cycles}(a)). 
Such intersecting sections are absent from $C'_2$.

If the intersection of $C_1$ and $C'_2$ is a nontrivial path, then the cycle $C'_2$ satisfies the claim.
If not, then we repeat the procedure described above for $C_1$ and $C'_2$ with the same numbering of $C_1$, and $v_k$ (or $v_j$) in the role of $v_i$. After a finite number of iterations, we obtain a cycle $C_3$ that satisfies the claim.
\end{proof}

From now on, given a graph $G$ and a resolving set $R$ of $G$, we will consider a subgraph of $G_R$ that will be ``conveniently chosen'' for our purposes. We consider a subgraph $H_R$ of the colour graph $G_R$ defined as follows. The vertex set of $H_R$ is $V(H_R) = V(G) \setminus R$ and the edge set $E(H_R)$ is a (possibly not proper) subset of $E(G_R)$ satisfying that there are either $0$ or $2$ edges of each colour $r\in R$ in $E(H_R)$, i.e., there are either $0$ or $2$ edges of each color in $H_R$.

Despite the fact that several possible configurations for such subgraph $H_R$ of $G_R$ could exist, we next show that the connected components of each $H_R$ are bipartite cacti.

\begin{lemma}
\label{lem:H-properties}
Let $G$ be a graph and let $R \subseteq V(G)$ be a resolving set of $G$. Then the following claims follow for $H_R$. 
\begin{enumerate}[label=\textnormal{(\roman*)}]
\item The graph $H_R$ is bipartite.
\item The connected components of $H_R$ are cacti.
\end{enumerate}
\end{lemma}

\begin{proof}
(i) Since $H_R$ is a subgraph of $G_R$, if $C$ is a cycle in $H_R$, then it is also a cycle in $G_R$. By Lemma~\ref{lem:colorprops}(i), a colour that appears in a cycle of $G_R$ appears at least twice in that cycle. Hence, by the choice of edges of $H_R$, a colour that appears in a cycle of $H_R$, appears exactly twice in that cycle. Therefore, as all edges have a colour, all cycles in $H_R$ have even length. It is well known that this implies that $H_R$ is bipartite.

\medskip
(ii) Suppose to the contrary that $H_R$ is not composed of cacti, in other words, that there exist cycles in $H_R$ that have edges in common. By Lemma~\ref{lem:cactus-theta}, we may assume that $C_1$ and $C_2$ are cycles such that their intersection is a path, denoted $P = C_1 \cap C_2$. Now consider any colour $r$ that appears on $P$. Clearly (as shown above), such colour $r$ must appear exactly twice in the cycle $C_1$, by the choice of the edges of $H_R$. If $r$ appears in $C_1\setminus P$, then $C_2$ is a cycle in which $r$ appears only once, a contradiction. Therefore, the second occurrence of $r$ must be in $P$ as well.

Denote the vertices of $C_1$ by $w_1, w_2, \dots, w_m$ so that vertices $w_{i}$ and $w_{i+1}$ are adjacent (with the notation that $w_{m+1} = w_{1}$) and assume that $V(P) = \{w_1, \ldots, w_k\}$ where $k>1$.  Now consider the sequence of distances $d_G(r, w_1), d_G(r, w_2), \dots, d_G(r, w_m), d_G(r, w_1)$ for some $r \in R$. A difference in consecutive distances $d_G(r, w_{i})$ and $d_G(r, w_{i+1})$ (say, $d_G(r, w_{i}) - d_G(r, w_{i+1}) = j \ne 0$) corresponds to an $r$-coloured edge $w_{i} w_{i+1}$ in $E(H_R)$. If consecutive distances are equal, the corresponding edge in $H_R$ has a colour other than $r$. By the choice of the edges of $H_R$, there can be at most two indices $i$ and $i'$ where consecutive distances differ. Since the first and last elements of the sequence are equal, we must have $d_G(r, w_{i}) = d_G(r, w_{i+1}) + j$ and $d_G(r, w_{i'}) = d_G(r, w_{i'+1}) - j$ for some $j \ne 0$.

For colours $s \in R$ that do not appear in $P$, we have $d_G(s, w_1) = d_G(s, w_2) = \cdots = d_G(s, w_k)$. For colours $r \in R$ that do appear in $P$, we have $d_G(r, w_1) = d_G(r, w_k) + j - j = d_G(r, w_k)$, since $r$ appears twice in $P$. Now the distinct vertices $w_1$ and $w_k$ have $d_G(u, w_1) = d_G(u, w_k)$ for all $u \in R$, contradicting the assumption that $R$ is a resolving set of $G$. Therefore, cycles in $H_R$ do not have edges in common, in other words, $H_R$ is composed of cacti.
\end{proof}

The following result regarding the largest number of edges in a graph whose components are bipartite cacti shall be needed.

\begin{lemma}[{\cite[Lemma~2]{Hernando19-LD}}]
    \label{lem:cactus-bound}
    If $H$ is a bipartite graph with $\abs{V(H)} \geq 4$ such that all of its connected components are cacti, then $\frac{3}{4} \abs{E(H)} + 1 \leq \abs{V(H)}$.
\end{lemma}

With the tools above in hand, we are then able to present our main theorem. By the notation $G[V(G)\setminus S]$ we mean the graph induced by the vertices of $G$ not including the set $S\subseteq V(G)$.

\begin{theorem}
\label{th:main-bound}
Let $G$ be a graph of order $n$ with $\bfv(G) \ge 1$. Then
    $$\bfv(G)\le \frac{2}{3}(n-\dim(G)-1).$$
\end{theorem}

\begin{proof}
Let $R \subseteq V(G)$ be a metric basis of $G$. Let $H_R$ be a subgraph of $G_R$ (as previously described) with vertices $V(H_R) = V(G) \setminus R$ and for the edges of $H_R$, we select two edges of $G_R$ with the colour $b$ for each basis forced vertex $b$. 
Indeed, at least two such edges exist in $G_R[V(G) \setminus R]$ by Lemma~\ref{lem:colorprops-more}(i) for each basis forced vertex.
By definition, $\abs{V(H_R)} = n - \dim(G)$ and $\abs{E(H_R)} = 2\bfv(G)$.
Since we assumed that $\bfv(G) \ge 1$, Proposition~\ref{prop:n-4} gives us that $\dim(G) \le n-4$, and consequently, $\abs{V(H_R)} \ge 4$.
By Lemma~\ref{lem:H-properties}, the graph $H_R$ is bipartite and its connected components are cacti.
We can now use Lemma~\ref{lem:cactus-bound} to get
\begin{equation*}
    \frac{3}{4} \cdot 2\bfv(G) + 1 \leq  n - \dim(G)
    \quad \Rightarrow \quad \bfv(G) \leq \frac{2}{3} (n - \dim(G) - 1),
\end{equation*}
which completes the proof.
\end{proof}

In~\cite{BagheriUnique16}, Bagheri \emph{et al.} studied those graphs with a unique metric basis, that is, graphs $G$ satisfying $\dim(G) = \bfv(G)$. Among other results, they showed that for any even $k \geq 2$ there exists a graph $G$ such that $k = \bfv(G) = \dim(G) = 2(n-1)/5$, where $n$ denotes the order of $G$. Furthermore, they stated as an open question whether the metric dimension could be larger with respect to $n$. In the following corollary, we answer the question by stating that this is not possible.
\begin{corollary}
\label{th:special-case-bound}
Let $G$ be a graph of order $n$. Then
\begin{equation} \label{eq:2-5-bound}
    \bfv(G)\le \left\lfloor \frac{2}{5}(n-1) \right\rfloor.    
\end{equation}
Moreover, if $\bfv(G) = \left\lfloor\frac{2}{5}(n-1)\right\rfloor$ and  $n \equiv 1 \text{ or } 4 \pmod{5}$, then $G$ has a unique metric basis.
\end{corollary}

\begin{proof}
Equation~\ref{eq:2-5-bound} follows immediately by the bound of Theorem~\ref{th:main-bound}, due to the fact that $\dim(G)\ge \bfv(G)$.

Assume that $G$ attains the bound, i.e. $\bfv(G) = \lfloor \frac{2}{5}(n-1) \rfloor$. Recall that $\bfv(G) \leq \dim(G)$.
Based on Theorem~\ref{th:main-bound}, it holds that $\dim(G) \leq n- \frac{3}{2}\bfv(G) - 1$. Therefore, we have 
\begin{equation} \label{Eq:upper_lower_unique}
\left\lfloor \frac{2}{5}(n-1) \right\rfloor = \bfv(G) \leq \dim(G) \leq \left\lfloor n - \frac{3}{2}\bfv(G) - 1 \right\rfloor \text.
\end{equation}

Suppose first that $n \equiv 1 \pmod{5}$, that is, let $n=5q+1$ for some $q \in \mathbb{N}$. Now $\bfv(G) = 2q$ and $\lfloor n - \frac{3}{2}2q - 1 \rfloor = 2q$. Thus, we must have $\dim(G) = \bfv(G)$, and so, $G$ has a unique metric basis. Suppose then that $n \equiv 4 \pmod{5}$ and denote $n=5q-1$ where $q\in \mathbb{N}$. As in the previous statement, $G$ has a unique metric basis since $\bfv(G) = \lfloor \frac{2}{5}(5q-2) \rfloor = 2q-1$ and $\lfloor n - \frac{3}{2}(2q-1) - 1 \rfloor = \lfloor 2q -\frac{1}{2} \rfloor = 2q-1$.
\end{proof}

\begin{table} 
	\centering
			\begin{tabular}{c|c|c}
				$n$     & $f = \lfloor 2(n-1)/5 \rfloor$ & $\lfloor n - 3f/2 - 1 \rfloor$ \\ \hline
				$5q$    &    $2q-1$  &    $2q$    \\
				$5q+1$  &    $2q$    &    $2q$    \\
				$5q+2$  &    $2q$    &    $2q+1$    \\
				$5q+3$  &    $2q$    &    $2q+2$    \\
				$5q+4$  &    $2q+1$    &    $2q+1$    \\
			\end{tabular}	
		\caption{The upper and lower bounds on $\dim(G)$ of Inequality~\eqref{Eq:upper_lower_unique} given for different remainders when divided by $5$.} \label{Table:nmod5}
	\end{table}

As stated above, if $\bfv(G) = \left\lfloor\frac{2}{5}(n-1)\right\rfloor$ and  $n \equiv 1 \text{ or } 4 \pmod{5}$, then $\dim(G) = \bfv(G)$ (and $G$ has a unique metric basis). However, this is not the case for $n \equiv 0, 2 \text{ or } 3 \pmod{5}$. By Table~\ref{Table:nmod5}, we observe that the following values are possible for $\dim(G)$: if $n \equiv 0 \text{ or } 2 \pmod{5}$, then $\dim(G) \in \{2|V(G)|/5,\ 2|V(G)|/5+1\}$, and $\dim(G) \in \{2|V(G)|/5,\ 2|V(G)|/5+1,\ 2|V(G)|/5+2\}$ for $n \equiv 3 \pmod{5}$. Later, in Theorem~\ref{thm:realization}, it will be shown that all these values can be realized, i.e., there exist graphs for all the possible combinations of $\dim(G)$ and $\bfv(G)$. 
For example, if $n = 5q + 3$, then (using the notation introduced in the proof of Theorem~\ref{thm:realization}) each of the graphs $G_{2q,0,2}$, $G_{2q,1,1}$, and $G_{2q,2,0}$ has $5q +3$ vertices of which $\lfloor 2(n-1)/5 \rfloor = 2q$ are basis forced, and their metric dimensions are $2q$, $2q +1$,  and $2q+2$, respectively.

\section{Realization of Possible Parameters}
\label{sec:realization}

In the previous sections, we have considered the relations between $|V(G)| = n$, $\dim(G)$ and $\bfv(G)$. In particular, we obtained the following lower and upper bounds for $\dim(G)$:
\begin{itemize}
    \item If $\bfv(G) = 1$, then $2 \leq \dim(G) \leq n-4$ by Proposition~\ref{prop:n-4}.
    \item If $\bfv(G) > 1$, then $\bfv(G) \leq \dim(G) \leq \left\lfloor n-\frac{3}{2}\bfv(G)-1 \right\rfloor$ by Theorem~\ref{th:main-bound}.
\end{itemize}
In what follows, we show that $\dim(G)$ achieves all the values indicated by the previous inequalities.

We build a graph step by step, first establishing the basis forced vertices, then building up the metric dimension, and finally, adding a path to reach the desired number of vertices.
In the construction, we need the concept of a \textit{universal vertex}, which is a vertex that is adjacent to all other vertices of the graph.
For the realization, we borrow two graphs from Bagheri et al. \cite{BagheriUnique16} and Buczkowski et al. \cite{BuczkowskiUnique}, illustrated in Figure~\ref{fig:2-3-bfvs}.
Formally, we define the graph
$\bfvgraph{2} = (V(\bfvgraph{2}), E(\bfvgraph{2}))$, where $V(\bfvgraph{2}) = \{x_1, x_2, y_0, y_1, y_2, y_3\}$ and $$E(\bfvgraph{2}) = \{y_i y_j \mid 0 \leq i < j \leq 3\} \cup \{x_1 y_1, x_1y_3, x_2y_2, x_2y_3 \}.$$
The graph $\bfvgraph{2}$ has a unique metric basis $\{x_1, x_2\}$ (as shown in~\cite{BuczkowskiUnique}).
Consequently, the vertices $x_1$ and $x_2$ are basis forced vertices.
The vertex $y_3$ is universal.

Then, we define $\bfvgraph{3} = (V(\bfvgraph{3}), E(\bfvgraph{3}))$, where $V(\bfvgraph{3}) = \{v_1, v_2, v_3, w_1, w_2, w_3, w_4, w_5, w_6\}$ and $$E(\bfvgraph{3}) = \{w_i w_j \mid 1 \leq i < j \leq 6\} \cup \{v_i w_j  \mid 1 \leq i \leq 3,\ j= i,\ i+1,\ i+2 \}.$$ The graph $\bfvgraph{3}$ has a unique metric basis $\{v_1, v_2, v_3\}$ (as shown in~\cite{BagheriUnique16}).
Consequently, the vertices $v_1$, $v_2$ and $v_3$ are basis forced vertices.
The vertex $w_3$ is universal.
The graphs $\bfvgraph{2}$ and $\bfvgraph{3}$ have a maximum number of basis forced vertices with respect to the order of the graph, meeting the bound of Corollary~\ref{th:special-case-bound}.

\begin{figure}
     \centering
     \begin{subfigure}[t]{0.45\textwidth}
        \centering
        \begin{tikzpicture}[yscale=1, xscale=1, rotate=0]

    \renewcommand*{\EdgeLineWidth}{ 0.5pt}
    \SetVertexMath
    \SetVertexLabelOut

    {
    \SetVertexNoLabel
    \tikzstyle{VertexStyle}=[minimum size=1pt,inner sep=0pt]

    \Vertices{circle}{2,3,1,0}
    \Vertex[x=-1, y=1.5]{11}
    \Vertex[x=1, y=1.5]{22}

    }

    \tikzstyle{LabelStyle}=[fill=white, circle]

    \tikzstyle{EdgeStyle}=[]
    \Edge(11)(3)
    \Edge(11)(1)
    \Edge(22)(2)
    \Edge(22)(3)
    
    \Edge(0)(1)
    \Edge(0)(2)
    \Edge(0)(3)
    \Edge(1)(2)
    \Edge(1)(3)
    \Edge(2)(3)

    \tikzstyle{VertexStyle}=[circle,draw, fill=black, inner sep=0pt, minimum width=6pt]
    \SO[unit=0, Lpos=90](11){x_1}
    \SO[unit=0, Lpos=90](22){x_2}

    \tikzstyle{VertexStyle}=[circle,draw, fill=white,inner sep=0pt, minimum width=6pt]
    \SO[unit=0, Lpos=90](3){y_3}
    \SO[unit=0, Lpos=0](2){y_2}
    \SO[unit=0, Lpos=180](1){y_1}
    \SO[unit=0, Lpos=-90](0){y_{0}}

\end{tikzpicture}

		\caption{A graph on $6$ vertices with $2$ basis forced vertices.}

     \end{subfigure}
     ~
     \begin{subfigure}[t]{0.45\textwidth}
        \centering
        \begin{tikzpicture}[yscale=-1.5, xscale=1, rotate=0]

    \renewcommand*{\EdgeLineWidth}{ 0.5pt}
    \SetVertexMath
    \SetVertexLabelOut

    {
    \SetVertexNoLabel
    \tikzstyle{VertexStyle}=[minimum size=1pt,inner sep=0pt]

    \Vertex[x=0, y=0]{11}
    \Vertex[x=1.5, y=0]{22}
    \Vertex[x=3, y=0]{33}

    \Vertex[x=0, y=1]{1}
    \Vertex[x=1, y=1]{2}
    \Vertex[x=2, y=1]{3}
    \Vertex[x=3, y=1]{4}
    \Vertex[x=4, y=1]{5}
    \Vertex[x=5, y=1]{6}

    }

    {
    \tikzstyle{VertexStyle}=[fill=white, inner sep = 1pt]
    \SetVertexLabelIn
    \Vertex[x=-0.7, y=1.1, L={K_6}]{Km}
    }

    \tikzstyle{LabelStyle}=[fill=white, circle]

    \tikzstyle{EdgeStyle}=[]
    \Edge(11)(1)
    \Edge(11)(2)
    \Edge(11)(3)

    \Edge(22)(2)
    \Edge(22)(3)
    \Edge(22)(4)

    \Edge(33)(3)
    \Edge(33)(4)
    \Edge(33)(5)

    \tikzstyle{VertexStyle}=[circle,draw, fill=black, inner sep=0pt, minimum width=6pt]
    \SO[unit=0, Lpos=90](11){v_1}
    \SO[unit=0, Lpos=90](22){v_2}
    \SO[unit=0, Lpos=90](33){v_3}

    \tikzstyle{VertexStyle}=[circle,draw, fill=white,inner sep=0pt, minimum width=6pt]
    \SO[unit=0, Lpos=-90](1){w_1}
    \SO[unit=0, Lpos=-90](2){w_2}
    \SO[unit=0, Lpos=-90](3){w_3}
    \SO[unit=0, Lpos=-90](4){w_4}
    \SO[unit=0, Lpos=-90](5){w_5}
    \SO[unit=0, Lpos=-90](6){w_6}

    \node[draw=black, rectangle, dashed, fit=(Km) (w_6), inner sep=6pt, rounded corners, thin] (machine) {};
    
\end{tikzpicture}

		\caption{A graph on $9$ vertices with $3$ basis forced vertices.}
     \end{subfigure}
        
     \caption{}
     \label{fig:2-3-bfvs}
\end{figure}

The following iterative construction can be used to obtain a graph with any number of basis forced vertices ($\geq 2$).

\begin{lemma}[{\cite[Theorem 7]{BagheriUnique16}}]
    \label{lem:recursive-construction}
    Let $G_1 = (V(G_1), E(G_1))$ and $G_2 = (V(G_2), E(G_2))$ be graphs with unique metric bases $R_1$ and $R_2$, respectively.
    Let $u_1 \in V(G_1)$ and $u_2 \in V(G_2)$ be universal vertices, i.e., vertices of degree $\abs{V(G_i)} - 1$, in their respective graphs.
    Then the graph $G = (V(G), E(G))$, where $V(G) =V(G_1) \cup V(G_2) \setminus \{u_2\}$ and $$E(G) = E(G_1) \cup E(G_2) \cup \{ u_1 w \mid w \in V(G_2)\}$$ has a unique metric basis $R_1 \cup R_2$. In addition, the vertex $u_1$ is a universal vertex in $G$. 
\end{lemma}
Notice that we may use Lemma~\ref{lem:recursive-construction} repeatedly to join any number of graphs that have a universal vertex and a unique metric basis.
If the target number of basis forced vertices $f$ is odd, then we take one copy of the graph $\bfvgraph{3}$ and $\frac{f-3}{2}$ copies of the graph $\bfvgraph{2}$ and join them by identifying their universal vertices.
If $f$ is even, then we join $\frac{f}{2}$ copies of the graph $\bfvgraph{2}$ in a similar manner.
A graph constructed in this way meets the upper bound of Corollary~\ref{th:special-case-bound}.

To increase the metric dimension of a graph with a universal vertex and a unique metric basis, we duplicate the universal vertex to obtain a sufficiently large set of twin vertices.

\begin{lemma}
    \label{lem:add-univ-vertex}
    Let $G$ be a graph with a unique metric basis $R$ and a universal vertex $u_0$.
    We define a graph $G'$ with $V(G') = V(G) \cup \{u_1, \dots, u_m\}$ and $$E(G') = E(G) \cup \{u_i w \mid w \in V(G), 1 \leq i \leq  m\} \cup \{u_i u_j \mid 1 \leq i < j \leq m\}.$$
    In other words, we add $m$ universal vertices.
    Then $G'$ has metric dimension $\dim(G) + m$ and its metric bases are $R \cup (U \setminus \{u_i\})$, where $U = \{u_0, u_1, \dots, u_m\}$ and $i \in \{0,1, \dots, m\}$.
\end{lemma}
\begin{proof}
By \cite[Lemma~17]{BasisForced}, we know that universal vertices are never forced.
    Since $R$ is the only metric basis of $G$, we know that $u_0 \notin R$. 
    Next, we observe that adding another universal vertex does not change the distances in a graph; in other words, $d_{G}(x,y) = d_{G'}(x,y)$ for all pairs of vertices $x, y \in V(G)$.
    Since a universal vertex $u$ has distance $1$ to all vertices other than itself, it cannot resolve pairs that do not contain $u$.
    The sets $R \cup (U \setminus \{u_i\})$ are resolving sets of $G'$ for all $i \in \{0, \dots, m\}$ and therefore, $\dim(G') \leq \dim(G)+m$.
    
    Let $R'$ be a metric basis of $G'$.
    Since the vertices $u_0, \dots, u_m$ are twins, at least $m$ of them must be in $R'$. 
    Let us first assume that all the vertices $u_0, \dots, u_m$ are in $R'$.
    Then the set $R' \setminus \{u_1, \dots, u_m\}$ is a resolving set of $G$ with $\dim(G') - m$ vertices, since no pair of vertices in $V(G)$ is resolved by the universal vertices $u_1, \dots, u_m$.
    We established that $\dim(G') \leq \dim(G)+m$, or equivalently, $\dim(G') -m \leq \dim(G)$, and hence, $R' \setminus \{u_1, \dots, u_m\}$ is a metric basis of $G$.
    However, we assumed that $G$ has a unique metric basis $R$, which does not contain $u_0$, and since $u_0 \in R' \setminus \{u_1, \dots, u_m\}$, we reached a contradiction.
    
    Let us now assume that exactly $m$ of the vertices $u_0, \dots, u_m$ are in $R'$, and denote $u_i \notin R'$, where $i \in \{0, \dots , m\}$.
    Next, we argue that the set $R' \setminus U$ is a resolving set of $G$ with $\dim(G')-m$ vertices.
    For this purpose, notice that for all $x \in V(G) \setminus \{u_0\}$ there exists a vertex $r' \in R'$ such that $d_{G'}(x,r') \neq 1 = d_{G'}(u_i, r')$.
    Due to the fact that all vertices of $U$ are universal, we obtain that $r' \in R' \setminus U$.
    Therefore, we also have $d_{G}(x,r') \neq 1 = d_{G}(u_0, r')$.
    Furthermore, it is immediate that distinct vertices $x$ and $y \in V(G) \setminus \{u_0\}$ are resolved by $R' \setminus U$. 
    Thus, $R' \setminus U$ is a resolving set of $G$.
    We established that $\dim(G') \leq \dim(G)+m$, or equivalently, $\dim(G') -m \leq \dim(G)$, and hence, $R' \setminus U$ is a metric basis of $G$.
    We assumed that $G$ has only one metric basis, namely $R$, so it must be that $R' \setminus U = R$.
    Therefore, the metric bases of $G'$ are $R \cup (U \setminus \{u_i\})$, where $i \in \{0, \dots, m\}$.
\end{proof}

Notice that the vertices in the unique metric basis remain basis forced in the new graph $G'$.

Finally, with $\bfv(G)$ and $\dim(G)$ fixed, we need to adjust $\abs{V(G)}$.
We do this by attaching a path of suitable length to a vertex in such a way that the number of basis forced vertices and the metric dimension of the graph does not change.

\begin{theorem}[{\cite[Theorem~2]{BasisForced-unic}}]
    \label{thm:attach-path}
    Let $G$ be a connected graph, and let $B \neq \emptyset$ be the set of basis forced vertices of $G$.
    Let $b \in B$, and let $v \in V(G) \setminus B$ be such that $d_G(b, v) = \max\{d_G(b, w) \mid w \in V (G) \setminus B\}$.
    Let $P_p$ be the path $v_1 \cdots v_p$.
    Let $H$ be the graph with $V(H) = V(G) \cup V (P_p)$ and $E(H) = E(G) \cup E(P_p) \cup \{vv_1\}$.
    Then, $B$ is also the set of basis forced vertices of $H$ and $\dim(H) = \dim(G)$.
\end{theorem}

With the tools above, we can prove the following theorem, which discusses the possible values of the parameters $|V(G)|$, $\dim(G)$ and $\bfv(G)$.

\begin{theorem} \label{thm:realization}
    Let $f$, $k$ and $n$ be positive integers.
    \begin{itemize}
        \item If $f=1$, then for each $k \in \{2, 3, \ldots, n-4\}$ there exists a graph $G$ such that $|V(G)| = n$, $\dim(G) = k$ and $\bfv(G) = f = 1$.
        \item If $f>1$, then for each $k \in \{f, f+1, \ldots, \left\lfloor n-\frac{3}{2}f-1 \right\rfloor\}$ there exists a graph $G$ such that $|V(G)| = n$, $\dim(G) = k$ and $\bfv(G) = f$.
    \end{itemize}
\end{theorem}

\begin{proof}
    We will provide a method to construct a graph $G$ with the desired values of $|V(G)|$, $\dim(G)$ and $\bfv(G)$.

    \begin{figure}[t]
     \centering
     \begin{subfigure}[t]{0.48\textwidth}
        \centering
        \begin{tikzpicture}[yscale=-1, xscale=0.9, rotate=0]

    \renewcommand*{\EdgeLineWidth}{ 0.5pt}
    \SetVertexMath
    \SetVertexLabelOut

    {
    \SetVertexNoLabel
    \tikzstyle{VertexStyle}=[minimum size=1pt,inner sep=0pt]

    \Vertex[x=0, y=0]{1}
    \Vertex[x=1.5, y=0]{2}
    \Vertex[x=3, y=0]{3}
    \Vertex[x=4, y=1]{4}

    \Vertex[x=0, y=2]{5}
    \Vertex[x=2, y=2]{6}
    \Vertex[x=3, y=2]{7}

    \Vertex[x=4.9, y=1]{8}
    \Vertex[x=6.3, y=1]{9}
    }

    {
    \tikzstyle{VertexStyle}=[fill=white, inner sep = 1pt]
    \SetVertexLabelIn
    \Vertex[x=1, y=2, L=\dots]{dots1}
    \Vertex[x=-0.7, y=2.15, L=\overline{K_m}]{Km}
    \Vertex[x=-0.7, y=-0.1, L=G_m]{Gm}
    }

    \tikzstyle{LabelStyle}=[fill=white, circle]

    \tikzstyle{EdgeStyle}=[]
    \Edge(1)(2)
    \Edge(1)(5)
    \Edge(1)(6)
    \Edge(1)(7)

    \Edge(2)(5)
    \Edge(2)(6)
    \Edge(2)(7)

    \Edge(3)(5)
    \Edge(3)(6)
    \Edge(3)(7)

    \Edge(4)(2)
    \Edge(4)(3)
    
    \Edge(4)(8)
    \Edge[label = \dots](8)(9)

    \tikzstyle{VertexStyle}=[circle,draw, fill=black, inner sep=0pt, minimum width=6pt]
    \SO[unit=0, Lpos=90](1){b}

    \tikzstyle{VertexStyle}=[circle,draw, fill=gray!60,inner sep=0pt, minimum width=6pt]
    \SO[unit=0, Lpos=-90](5){u_1}
    \SO[unit=0, Lpos=-90](6){u_{m-1}}
    \SO[unit=0, Lpos=-90](7){u_m}

    \tikzstyle{VertexStyle}=[circle,draw, fill=white,inner sep=0pt, minimum width=6pt]
    \SO[unit=0, Lpos=90](2){v_1}
    \SO[unit=0, Lpos=90](3){v_2}
    \SO[unit=0, Lpos=90](4){v_3}

    {
    \SetVertexNoLabel
    \SO[unit=0](8){p1}
    \SO[unit=0](9){p2}
    }

    \node[draw=black, rectangle, dashed, fit=(Km) (u_m), inner sep=6pt, rounded corners, thin] (machine) {};

    \node[draw=black, rectangle, dotted, fit=(Gm) (Km) (v_3), inner sep=9pt, rounded corners, thick] (machine) {};

    \draw [decorate, 
    decoration = {brace, raise=5pt, amplitude = 5pt}] (4.75,1) --  (6.45,1) node[pos=0.5,above=10pt,black]{$p$};
\end{tikzpicture}

        \caption{The graph $G_{m}$ is within the dotted line. The graph $G_{m,p}$ is obtained by attaching a path to $v_3$.}
     \end{subfigure}
     ~
     \begin{subfigure}[t]{0.48\textwidth}
        \centering
        \begin{tikzpicture}[yscale=-.5, xscale=1, rotate=90]

    \renewcommand*{\EdgeLineWidth}{ 0.5pt}
    \SetVertexMath
    \SetVertexLabelOut

    {
    \SetVertexNoLabel
    \tikzstyle{VertexStyle}=[minimum size=1pt,inner sep=0pt]

    \Vertex[x=0, y=0]{11}
    \Vertex[x=1.5, y=0]{22}
    \Vertex[x=3, y=0]{33}

    \Vertex[x=0, y=1]{1}
    \Vertex[x=1, y=1]{2}
    \Vertex[x=2, y=1]{3}
    \Vertex[x=3, y=1]{4}
    \Vertex[x=4, y=1]{5}
    \Vertex[x=5, y=1]{6}
    
    \Vertex[x=5, y=0.25]{8}
    \Vertex[x=5, y=-1]{9}

    \Vertex[x=2, y=2.1]{u1}
    \Vertex[x=2, y=3]{u2}
    }

    {
    \tikzstyle{VertexStyle}=[fill=white, inner sep = 1pt]
    \SetVertexLabelIn
    \Vertex[x=2, y=2.5, L=\dots]{dots1}
    \Vertex[x=-0.83, y=1.3, L={K_6}]{Km}
    }

    \tikzstyle{LabelStyle}=[fill=white, circle]

    \tikzstyle{EdgeStyle}=[]
    \Edge(11)(1)
    \Edge(11)(2)
    \Edge(11)(3)

    \Edge(22)(2)
    \Edge(22)(3)
    \Edge(22)(4)

    \Edge(33)(3)
    \Edge(33)(4)
    \Edge(33)(5)

    \Edge(6)(8)
    \Edge[label=\dots](8)(9)

    \tikzstyle{VertexStyle}=[circle,draw, fill=black, inner sep=0pt, minimum width=6pt]
    \SO[unit=0, Lpos=0](11){v_1}
    \SO[unit=0, Lpos=0](22){v_2}
    \SO[unit=0, Lpos=0](33){v_3}

    \tikzstyle{VertexStyle}=[circle,draw, fill=white,inner sep=0pt, minimum width=6pt]
    \SO[unit=0, Lpos=180](1){w_1}
    \SO[unit=0, Lpos=180](2){w_2}
    \SO[unit=0, Lpos=180](4){w_4}
    \SO[unit=0, Lpos=180](5){w_5}
    \SO[unit=0, Lpos=180](6){w_6}
    
    {
    \SetVertexNoLabel
    \SO[unit=0](8){p1}
    \SO[unit=0](9){p2}
    }
    
    \tikzstyle{VertexStyle}=[circle,draw, fill=gray!60,inner sep=0pt, minimum width=6pt]
    \SO[unit=0, Lpos=-90, Ldist=5](u1){u_1}
    \SO[unit=0, Lpos=-90, Ldist=5](u2){u_m}
    \SO[unit=0, Lpos=180](3){w_3}

    \node[draw=black, rectangle, dashed, fit=(Km) (w_6), inner sep=6pt, rounded corners, thin] (machine) {};
    \node[draw=black, rectangle, dashed, fit=(u2) (w_3), inner sep=6pt, rounded corners, thin] (machine) {};

    \draw [decorate, 
    decoration = {brace, raise=5pt, amplitude = 5pt}] (5,0.4) --  (5,-1.15) node[pos=0.5,above=10pt,black]{$p$};
\end{tikzpicture}

        \caption{The graph $G_{f, m, p}$, with $f = 3$. The edges from the universal vertices $u_1, \dots, u_m$ are omitted for illustrative purposes.}
     \end{subfigure}
     \caption{Sketches of the constructions.  The black vertices are basis forced and the gray vertices are in some but not all metric bases.}
     \label{fig:realization}
\end{figure}
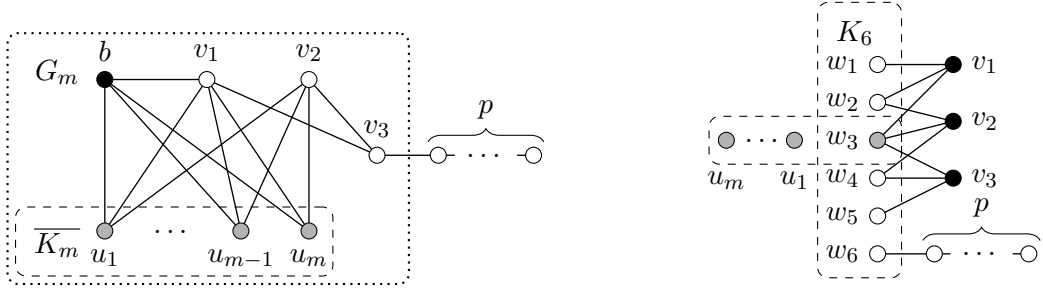
    
    First, we consider the case $f = 1$. By Proposition~\ref{prop:n-4}, we know that $2 \le \dim(G) \leq n - 4$. Consider the graph $G_{m}$ illustrated in Figure~\ref{fig:realization}(a), where $m \ge 2$.  Since the $m$ vertices in $\overline{K_m}$ are twins, a metric basis must contain at least $m-1$ of them.
    The set $V(\overline{K_m}) \setminus\{u_j\}$ is not a resolving set for any $j \in \{1, \dots, m\}$ since the vertices $v_1$ and $v_2$ are not resolved. 
    It follows that $\dim(G_m) > m-1$.
    It is easy to verify that $V(\overline{K_m})[u_j \leftarrow b]$ 
    is a resolving set of $G_m$ for all $j \in \{1, \dots, m\}$.
    The sets $V(\overline{K_m})$,  $V(\overline{K_m})[u_j \leftarrow v_3]$ and  $V(\overline{K_m})[u_j \leftarrow v_i]$ 
    are not resolving sets for any choice of $j$ or $i \in \{1,2\}$ since the pairs $\{v_1, v_2\}$, $\{v_1, v_2\}$ and $\{u_j, v_3\}$ are not resolved, respectively.
    Therefore, $\dim(G_m) = m$, the vertex $b$ is basis forced, and the vertex $v_3$ is not basis forced and has $d_G(b, v_3) = 2 = \max\{d_G(b, w) \mid w \in V(G_m)\}$.
    Choosing $m = k$, and attaching a path of length $p = n - k - 4$ to the vertex $v_3$, we get the graph $G_{m,p}$ with $m + 4 + p = n$ vertices (see Figure~\ref{fig:realization}(a)).
    By Theorem~\ref{thm:attach-path}, the graph $G_{m,p}$ has the same metric dimension and basis forced vertices as $G_{m}$, and therefore, the graph $G_{m,p}$ has metric dimension $m = k$ and one basis forced vertex (namely, $b$).
    We assumed that $m \geq 2$ and, naturally, that $p \geq 0$. 
    Hence, when $m$ and $p$ are chosen as defined, this construction realizes the values of $k$ and $n$ such that $k \geq 2$ and $n - k - 4 \geq 0$.

    Now assume that $f > 1$.
    If $f$ is odd, then ($f \geq 3$ and) we take one copy of the graph $\bfvgraph{3}$ and $\frac{f-3}{2}$ copies of the graph $\bfvgraph{2}$ and join them by identifying their universal vertices. (Note that if $f =3$, then there are no copies of $\bfvgraph{2}$ in the graph.)
    If $f$ is even, then we join $\frac{f}{2}$ copies of the graph $\bfvgraph{2}$ in a similar manner.
    Let us call the universal vertex of the resulting graph $u$.
    So far, our graph has $\frac{5f+1}{2}+1$ vertices if $f$ is odd and $\frac{5}{2}f+1$ vertices if $f$ is even; to simplify the notation ahead, we denote this number of vertices by $\h$.
    By Lemma~\ref{lem:recursive-construction}, the basis forced vertices of the joined graphs form the unique metric basis of the resulting graph and there are $f$ of them.

    Next, we duplicate the universal vertex $u$ in the spirit of Lemma~\ref{lem:add-univ-vertex} $m$ times.
    The metric dimension of our graph is now $f + m$.
    Finally, we attach a path of $p$ vertices to some vertex $y_0$ in a copy of $\bfvgraph{2}$ (or $w_6$, in case $f=3$ and there are no copies of $\bfvgraph{2}$), resulting in a graph denoted by $G_{f, m, p}$.
    Theorem~\ref{thm:attach-path} guarantees that the graph $G_{f, m, p}$ has metric dimension $f + m$ and that the $f$ basis forced vertices remain basis forced.
    
    Choosing $m = k - f$ and $p = n - \h - (k - f)$
    gives us the graph $G_{f, m, p}$ with $\h + m + p = n$
    vertices, $\dim(G_{f, m, p}) = f + m = k$, and $\bfv(G_{f, m, p}) = f$.
    The construction requires that $m \geq 0$ and $p \geq 0$.
    The choices of $m$ and $p$ are possible if and only if $k - f \geq 0$ and $n - \h - (k - f) \geq 0$, or equivalently, when $f \leq k \leq n - (\h - f)$.
    When we substitute $\frac{5f+1}{2}+1$ and $\frac{5}{2}f+1$ for $\h$ (depending on the parity of $f$), we get the upper bounds $k \leq n - \frac{3f+1}{2}-1$ and $k \leq n - \frac{3}{2}f-1$, which can be combined into $k \leq \left\lfloor n-\frac{3}{2}f-1 \right\rfloor $,
    as claimed.
\end{proof}

\section{A Characterization of Graphs With Two Basis Forced Vertices and $\dim (G) = n-4$}\label{CharSec}

In the previous section, we proved that the upper bounds we have established are tight since they are attained by particular graphs. This then leads to the natural, but more difficult question, about precisely which graphs are extremal in relation to these bounds. According to Proposition~\ref{prop:n-4}, the universal upper bound on the metric dimension is $n-4$, when basis forced vertices are present. However, due to Theorem~\ref{th:main-bound}, this can be attained only when either $\bfv(G) = 1$ or $\bfv(G) = 2$. Thus, the case where $\bfv(G)=2$ and $\dim(G) = n-4$ is where both the number of basis forced vertices and the metric dimension are extremal.

In this section, we fully characterize the graphs containing exactly two basis forced vertices with metric dimension $n-4$. We searched for small such graphs with a computer and found out that the smallest among them have six vertices. Through an exhaustive search of graphs of order 6 (based on the data available on Brendan McKay's website~\cite{McKayCombData}), we were able to establish that the graphs $G_1$, $G_2$ and $G_3$ in Figure~\ref{fig-smallest} are the only graphs with $n=6$, $\dim (G) = n-4 = 2$ and two basis forced vertices.

Let $\F$ be a graph family defined as follows. A graph $G\in \F$ if either $G$ is the graph $G_3$ from Figure~\ref{fig-smallest}, or $G$ contains a clique $K_{n-2}$ and the vertices of the clique can be partitioned into four nonempty sets $V_i$, $i = 1, \ldots, 4$ satisfying the following two properties: (i) the two vertices $u$ and $v$ that are not in the clique are such that $u$ is adjacent to the vertices in $V_1 \cup V_2$ but not to the vertices in $V_3 \cup V_4$, and $v$ is adjacent to the vertices in $V_1 \cup V_3$ but not to the vertices in $V_2 \cup V_4$; (ii) the vertices $u$ and $v$ may or may not be adjacent to each other (see Figure~\ref{fig-family} for a representative example of a graph in $\F$).  The graphs $G_1$ and $G_2$ from $\F$ with the smallest order are given in Figure~\ref{fig-smallest}. 
We next give the characterization of graphs with $\bfv(G)=2$ and $\dim (G) = n-4$.

\begin{figure}
	\centering
	\begin{subfigure}[b]{0.3\linewidth}
		\centering
		\begin{tikzpicture}[scale=1]
			\draw (0,0) -- (1,0) -- (1,1) -- (0,1) -- (0,0);
			\draw (0,0) -- (1,1)  (1,0) -- (0,1);
			\draw (0,0) -- (-1,.5) -- (0,1) -- (.5,2) -- (1,1);
			\draw \foreach \x in {(0,0),(1,0),(0,1),(1,1)} {
				\x node[circle, draw, fill=white,
				inner sep=0pt, minimum width=7pt] {}
			};
			\draw \foreach \x in {(-1,.5),(.5,2)} {
				\x node[circle, draw, fill=black,
				inner sep=0pt, minimum width=7pt] {}
			};
			\draw 
				(-1,.5) node[left, inner sep=7pt] {$v$}
				(.5,2) node[right, inner sep=7pt] {$u$};
		\end{tikzpicture}
		\caption{$G_1$} 
	\end{subfigure}
	\hfill
	\begin{subfigure}[b]{0.3\linewidth}
		\centering
		\begin{tikzpicture}[scale=1]
			\draw (0,0) -- (1,0) -- (1,1) -- (0,1) -- (0,0);
			\draw (0,0) -- (1,1)  (1,0) -- (0,1);
			\draw (0,0) -- (-1,.5) -- (0,1) -- (.5,2) -- (1,1);
			\draw (-1,.5) to[out=80,in=190,looseness=1] (.5,2);
			\draw \foreach \x in {(0,0),(1,0),(0,1),(1,1)} {
				\x node[circle, draw, fill=white,
				inner sep=0pt, minimum width=7pt] {}
			};
			\draw \foreach \x in {(-1,.5),(.5,2)} {
				\x node[circle, draw, fill=black,
				inner sep=0pt, minimum width=7pt] {}
			};
			\draw 
			(-1,.5) node[left, inner sep=7pt] {$v$}
			(.5,2) node[right, inner sep=7pt] {$u$};
		\end{tikzpicture}
		\caption{$G_2$} 
	\end{subfigure}
	\hfill
	\begin{subfigure}[b]{0.3\linewidth}
		\centering
		\begin{tikzpicture}[scale=1]
			\draw (0,0) -- (1,0) -- (1,1) -- (0,1) -- (0,0);
			\draw (0,0) -- (1,1) ;
			\draw (0,0) -- (-1,.5) -- (0,1) -- (.5,2) -- (1,1);
			\draw (-1,.5) to[out=80,in=190,looseness=1] (.5,2);
			\draw \foreach \x in {(0,0),(1,0),(0,1),(1,1)} {
				\x node[circle, draw, fill=white,
				inner sep=0pt, minimum width=7pt] {}
			};
			\draw \foreach \x in {(-1,.5),(.5,2)} {
				\x node[circle, draw, fill=black,
				inner sep=0pt, minimum width=7pt] {}
			};
			\draw 
			(-1,.5) node[left, inner sep=7pt] {$v$}
			(.5,2) node[right, inner sep=7pt] {$u$};
		\end{tikzpicture}
		\caption{$G_3$} 
	\end{subfigure}
\caption{The smallest graphs with $\dim(G) = n-4$ and two basis forced vertices (illustrated in black).}\label{fig-smallest}
\end{figure}
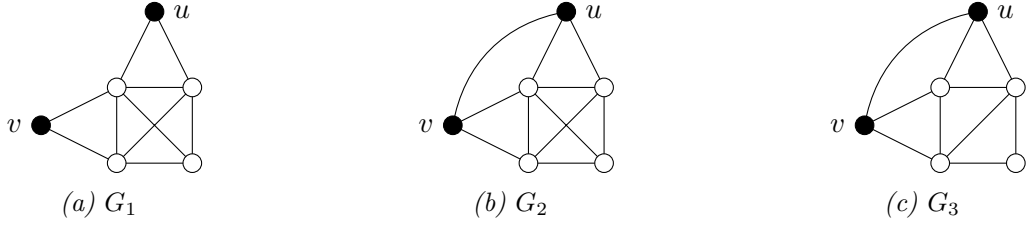

\begin{figure}[b]
	\centering
		\begin{tikzpicture}[scale=1.5]
			\coordinate (v11) at (-.2,.8);
			\coordinate (v12) at (.2,1.2);
			\coordinate (v21) at (.8,1.2);
			\coordinate (v22) at (1.2,.8);
			\coordinate (v3) at (0,0);
			\coordinate (v4) at (1,0);
			\coordinate[label=180:$u \ $] (u) at (.5,2);
			\coordinate[label=180:$v \ $] (v) at (-1,.5);
			\draw \foreach \x in {(v11),(v12),(v21),(v22)}{(u) -- \x};
			\draw \foreach \x in {(v11),(v12),(v3)}{(v) -- \x};
			\draw \foreach \x in {(v11),(v12),(v21),(v22),(v3),(v4)}{
				\foreach \y in {(v11),(v12),(v21),(v22),(v3),(v4)}{\x--\y}};
			\draw \foreach \x in {(v11),(v12),(v21),(v22),(v3),(v4)} {
				\x node[circle, draw, fill=white,
				inner sep=0pt, minimum width=7pt] {}
			};
			\draw \foreach \x in {(v),(u)} {
				\x node[circle, draw, fill=black,
				inner sep=0pt, minimum width=7pt] {}
			};
			\draw[dashed] (-.4,-.4) rectangle (1.4,1.4);
			\draw[dashed] (-.4,.5) -- (1.4,.5)  (.5,-.4) -- (.5,1.4);
			\draw (-.6,1.2) node[] {$V_1$};
			\draw (1.7,.9) node[] {$V_2$};
			\draw (0,-.7) node[] {$V_3$};
			\draw (1,-.7) node[] {$V_4$};
		\end{tikzpicture}
		\caption{An example graph $G$ from the family $\F$.}\label{fig-family}
\end{figure}
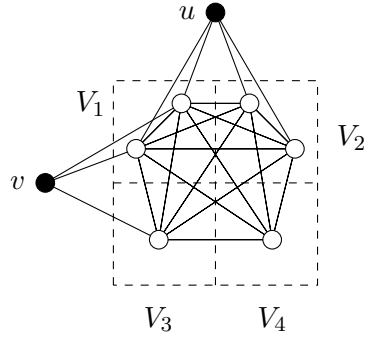

\begin{theorem}\label{char2bf}
Let $G$ be a connected graph with $n \geq 6$. Then $\dim(G) = n-4$ and $G$ contains two basis forced vertices if and only if $G \in \F$. Moreover, the two basis forced vertices of $G$ are $u$ and $v$, as denoted in the definition of the family $\F$.
\end{theorem}

\begin{proof}	
For the purposes of this proof, we define a \emph{distance\nobreakdash-2 partition} of a graph with respect to two vertices $x$ and $y$ as follows. The sets $V_{11}$, $V_{12}$, $V_{21}$ and $V_{22}$ form a distance\nobreakdash-2 partition of $G$ with respect to $x$ and $y$ if
\begin{itemize}
		\item $V(G) \setminus \{x,y\} = V_{11} \cup V_{12} \cup V_{21} \cup V_{22}$,
		\item $V_{ij} \neq \emptyset$ for all $i,j \in \{1,2\}$,
		\item if $v \in V_{ij}$, then $d(x,v) = i$ and $d(y,v) = j$.
	\end{itemize}
		
	\medskip

\noindent
\textbf{Observation 1:} If $G$ can be distance\nobreakdash-2 partitioned with respect to two vertices $x$ and $y$, then any set $R = V(G) \setminus \{ v_{11}, v_{12}, v_{21}, v_{22} \}$ where $v_{11} \in V_{11}$, $v_{12} \in V_{12}$, $v_{21} \in V_{21}$ and $v_{22} \in V_{22}$ is a resolving set of $G$.

\smallskip\noindent
\textit{Proof of Observation 1:} All vertices in $R$ are clearly resolved from all other vertices in $G$. The vertices $v_{11}$, $v_{12}$, $v_{21}$ and $v_{22}$ are all in different partition sets, and thus $x$ and $y$ resolve them from each other. Therefore, $R$ is a resolving set of $G$. ($\diamond$)

\smallskip\noindent
We next proceed to show both implications in our result.

\smallskip\noindent
($\Leftarrow$) Assume that $G \in \F$. We will prove that $\dim (G) = n-4$ and that $G$ contains two basis forced vertices, which indeed are the vertices $u$ and $v$ from the definition of the graph $G\in\F$.
	
If $G=G_3$, then the claim holds according to an exhaustive computer search (the only metric basis of this graph is $\{u,v\}$ in Figure~\ref{fig-smallest}(c)).
	
Assume that $G \neq G_3$ and let $R$ be a metric basis of $G$. Since $G$ can clearly be distance\nobreakdash-2 partitioned with respect to $u,v$ (the vertices used in the definition of $G\in \F$), we have $\dim (G) \leq n-4$ due to Observation~1, and thus $|R| \leq  n-4$. The vertices within each partition set $V_i$ are twins, and thus at most one vertex from each set is not in~$R$. Suppose that $u \notin R$. If also $v \notin R$, then all elements of $R$ are within the clique $K_{n-2}$. Since $|R| \leq n-4$, there are two vertices in the clique that are not in~$R$. However, these vertices are not resolved from each other by any vertex of $R$, a contradiction. Thus, we can assume that $v \in R$. Since $|R| \leq n-4$, the sets $V_1$ and $V_3$ both contain a vertex that is not in $R$ or the sets $V_2$ and $V_4$ both contain a vertex that is not in $R$. 
This pair is not resolved by $v$ due to $v$ being either adjacent to or at distance 2 from both of them. Neither is this pair resolved by any vertex in $R$ that is in the clique $K_{n-2}$, since the pair is also included in the same clique. Thus, there is a pair of vertices that $R$ does not resolve, a contradiction. Hence, $u \in R$ (and $v \in R$ by similar arguments). Since each partition set contains at most one vertex that is not in~$R$, we now have $|R| \geq n-4$. Therefore, $\dim (G) = n-4$. Moreover, observe that any metric basis of $G$ is formed by $n-6$ vertices from the clique $K_{n-2}$, together with the two vertices $u,v$, which are hence basis forced vertices of $G$, which completes the proof of this implication. 

\bigskip
\noindent	
($\Rightarrow$) Let $G$ be a graph with $\dim(G) = n-4$ and such that it contains two basis forced vertices~$b_1$ and~$b_2$.
	
\medskip
\noindent
\textbf{Claim 1:} $G$ can be distance\nobreakdash-2 partitioned with respect to $b_1$ and $b_2$.

\smallskip
\noindent
\textit{Proof of Claim 1:} Let $R$ be a metric basis of $G$. Consider the graph $G_R$. Due to the case (i) of Lemma~\ref{lem:colorprops-more} and the cases (i) and (ii) of Lemma~\ref{lem:colorprops}, we obtain that $G_R [V(G) \setminus R] = C_4$ where the labels $b_1$ and $b_2$ alternate. Let us denote $V(G) \setminus R = \{x_1,x_2,y_1,y_2\}$ and let the edges of the aforementioned $C_4$ be $x_1x_2$, $x_2y_2$, $y_1y_2$ and $x_1y_1$ where $x_1x_2$ and $y_1y_2$ have the label $b_1$ and $x_1y_1$ and $x_2y_2$ have the label $b_2$ (see Figure~\ref{fig-colorgraph}). Notice that the edges $x_1y_2$ and $x_2y_1$ do not appear in $G_R$ as the corresponding pairs of vertices are resolved by both $b_1$ and $b_2$. Hence, for each $r \in R \setminus \{b_1,b_2\}$ (if such a vertex exists), there is exactly one edge of label $r$ in $G_R$, namely the edge $rw$ where $w \notin R$ (due to Lemma~\ref{lem:colorprops}~(ii)). 
	
	\begin{figure}
		\centering
		\begin{tikzpicture}[scale=1]
			\coordinate[label=180:$b_1 \ $] (b1) at (-3,1);
			\coordinate[label=180:$b_2 \ $] (b2) at (-3,-1);
			\coordinate (r1) at (-2,2);
			\coordinate (r2) at (-2,1);
			\coordinate (r3) at (-2,0);
			\coordinate (r4) at (-2,-1);
			\coordinate (r5) at (-2,-2);
			\coordinate (x1) at (0,1);
			\node at (x1) [above=1.5mm] {$x_1$};
			\coordinate (x2) at (2,2);
			\node at (x2) [right=1.5mm] {$x_2$};
			\coordinate (y1) at (0,-2);
			\node at (y1) [below=1.5mm] {$y_1$};
			\coordinate (y2) at (2,-1);
			\node at (y2) [right=1.5mm] {$y_2$};
			\draw[line width=0.3mm,red,style={decorate, decoration=snake}] (x1) to node[midway, below right, black] {$b_1$} (x2);
			\draw[line width=0.3mm,red,style={decorate, decoration=snake}] (y1) to node[midway, below right, black] {$b_1$} (y2);
			\draw[line width=0.3mm,blue] (x1) to node[midway, left, black] {$b_2$} (y1);
			\draw[line width=0.3mm,blue] (x2) to node[midway, right, black] {$b_2$} (y2);
			\draw[very thick,dotted] (r1) to (x2);
			\draw[very thick,dotted] (r2) to (x1) to (r3);
			\draw[very thick,dotted] (r4) to (y1) to (r5);
			\draw \foreach \x in {(x1),(x2),(y1),(y2)} {
				\x node[circle, draw, fill=white,
				inner sep=0pt, minimum width=7pt] {}
			};
			\draw \foreach \x in {(r1),(r2),(r3),(r4),(r5)} {
				\x node[circle, draw, fill=black,
				inner sep=0pt, minimum width=7pt] {}
			};
			\draw \foreach \x in {(b1)} {
				\x node[circle, draw, fill=red,
				inner sep=0pt, minimum width=7pt] {}
			};
			\draw \foreach \x in {(b2)} {
				\x node[circle, draw, fill=blue,
				inner sep=0pt, minimum width=7pt] {}
			};
			\draw[dashed] (-1,-3.3) -- (-1,3);
			\draw 
				(-2,-3) node[] {$R$}
				(.4,-3) node[] {$V(G) \setminus R$};
		\end{tikzpicture}
		\caption{The colour graph $G_R$.}\label{fig-colorgraph}
	\end{figure}
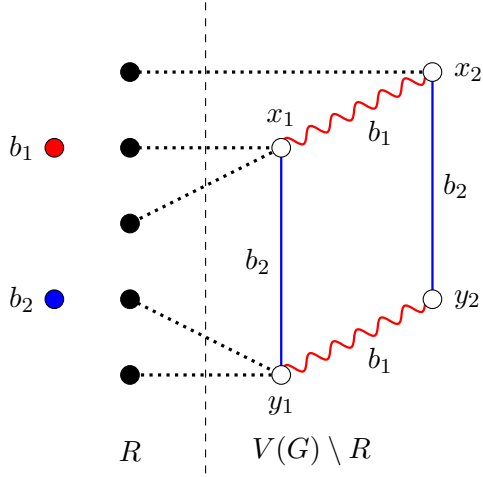
	
Let us consider the distances from $b_1$ to the vertices not in $R$.
Since the edge $x_1y_1$ has the label $b_2$ in $G_R$, the only vertex in $R$ that resolves $x_1$ and $y_1$ is $b_2$. In particular, this means that $d(b_1,x_1) = d(b_1,y_1)$. Similarly, since the edge $x_2y_2$ has the label $b_2$ in $G_R$, we have $d(b_1,x_2) = d(b_1,y_2)$. The edges $x_1x_2$ and $y_1y_2$ have, however, the label $b_1$, which means that $d(b_1,x_1) \neq d(b_1,x_2)$ and $d(b_1,y_1) \neq d(b_1,y_2)$. Therefore, the vertices not in $R$ can be partitioned into two sets $\{x_1,y_1\}$ and $\{x_2,y_2\}$ where vertices in the same set are at the same distance from $b_1$ but vertices in different sets are at different distances from $b_1$.
	
Let us then consider the distances from $b_1$ to the vertices in $R \setminus \{b_1,b_2\}$.
As stated before, the only edge with the label $r$ in $G_R$ where $r\in R \setminus \{b_1,b_2\}$ is of the form $rw$ where $w \notin R$. This means that the vertex $r$ is the only vertex in $R$ that resolves $r$ and $w$, and thus $d(b_1,r) = d(b_1,w)$. By combining this with the partition of the vertices not in $R$, we obtain the following partition for $V(G) \setminus \{b_1,b_2\} = V_1 \cup V_2$:
\begin{align*}
		V_1 &= \{x_1,y_1\} \cup \{r\in R \ | \ d(b_1,r) = d(b_1,x_1) = d(b_1,y_1)\}, \\
		V_2 &= \{x_2,y_2\} \cup \{r\in R \ | \ d(b_1,r) = d(b_1,x_2) = d(b_1,y_2)\}.
	\end{align*}

We will then consider the exact distances from $b_1$ to the partition sets $V_1$ and $V_2$. Consider the vertices $x_1 \in V_1$ and $x_2 \in V_2$. One of the partition sets is closer to $b_1$ than the other. Let us assume without loss of generality that $d(b_1,x_1) < d(b_1,x_2)$. We will prove that now $d(b_1,x_1) = 1$ and $d(b_1,x_2) = 2$. 
	
Suppose to the contrary that $d(b_1,x_1) > 1$. The (unique) shortest path from $b_1$ to $x_1$ is now necessarily $b_1b_2x_1$ since all other vertices in the graph are at equal or greater distance from $b_1$ compared to $x_1$. Now $d(b_2,x_1) = 1$ and, since the edge $x_1x_2$ has label $b_1$ in $G_R$, also $d(b_2,x_2) = 1$. This implies that the path $b_1b_2x_2$ is also present in $G$, which contradicts the fact that $d(b_1,x_1) \neq d(b_1,x_2)$. Therefore, $d(b_1,x_1) = 1$ and consequently $d(b_1,v) = 1$ for all $v \in V_1$.
	
Suppose then that $d(b_1,x_2) > 2$. Since the vertices in $V_1$ are all at distance~1 from $b_1$, a shortest path from $b_1$ to $x_2$ is of the form $b_1 w_1 b_2 x_2$, where $w_1 \in V_1$. Recall that $y_2 \in V_2$. Now there is a shortest path $b_1w_2b_2y_2$, where $w_2 \in V_1$. Thus, both $x_2$ and $y_2$ are at distance~1 from $b_2$. However, the edge $x_2y_2$ has the label $b_2$ in $G_R$, and thus $d(b_2,x_2) \neq d(b_2,y_2)$, a contradiction. Therefore, $d(b_1,x_2) = 2$ and consequently $d(b_1,w') = 2$ for all $w' \in V_2$.
	
Thus, the partition sets $V_1$ and $V_2$ are simply the vertices at distance~1 and~2 (respectively) from $b_1$. 
	
A similar partitioning can be done with respect to $b_2$. By similar arguments as for $b_1$, we can partition $V(G) \setminus \{b_1,b_2\}$ into $W_1$ and $W_2$ where
	\begin{align*}
		W_1 &= \{x_1,x_2\} \cup \{r\in R \ | \ d(b_2,r) = d(b_2,x_1) = d(b_2,x_2)\}, \\
		W_2 &= \{y_1,y_2\} \cup \{r\in R \ | \ d(b_2,r) = d(b_2,y_1) = d(b_2,y_2)\}.
	\end{align*}
Notice that now $x_1$ and $y_1$ (and $x_2$ and $y_2$) are in different partition sets. Without loss of generality, the set $W_1$ is closer to $b_2$ than $W_2$. By similar arguments as for $V_1$ and $V_2$, the vertices in $W_1$ are at distance~1 from $b_2$ and the vertices in $W_2$ are at distance~2 from $b_2$.
	
By combining the two partitions, we obtain a distance\nobreakdash-2 partition where $V_{11}  = V_1 \cap W_1$, $V_{12}  = V_1 \cap W_2$, $V_{21}  = V_2 \cap W_1$ and $V_{22}  = V_2 \cap W_2$. Notice that each of the sets $V_{11}$, $V_{12}$, $V_{21}$ and $V_{22}$ is nonempty since $x_1 \in V_{11}$, $y_1 \in V_{12}$, $x_2 \in V_{21}$ and $y_2 \in V_{22}$. ($\diamond$)
	
\medskip
Thus, we will consider the distance\nobreakdash-2 partition of $G$ with respect to $b_1$ and $b_2$ for the remainder of the proof.
    
\medskip
\noindent
\textbf{Claim 2:} If $|V_{ij}| \geq 2$ for some $i,j \in \{1,2\}$, then each $w \in V_{ij}$ is adjacent to all vertices in $V(G) \setminus \{b_1,b_2,w\}$.

\smallskip
\noindent
\textit{Proof of Claim 2:} Since $\dim (G) = n-4$, the set $R = V(G) \setminus \{ v_{11}, v_{12}, v_{21}, v_{22} \}$, where $v_{ij} \in V_{ij}$, is a metric basis of $G$ due to Observation~1. Let $V_{ij}$ be a partition set such that $|V_{ij}| \geq 2$. Now $V_{ij}$ contains an element of $R$, say, $r \in V_{ij} \cap R$. Due to the labels of the edges in the cycle in $G_R [V(G) \setminus R]$, we have 
    \begin{equation}\label{eqdist}
        d(r,v_{11}) = d(r,v_{12}) = d(r,v_{22}) = d(r,v_{21}).
    \end{equation} 
According to Observation~1, we can choose the vertices $v_{11}$, $v_{12}$, $v_{21}$ and $v_{22}$ to be any vertices in their respective partition sets. Therefore, the vertex $r$ is adjacent to either all or none of the vertices in $V(G) \setminus \{b_1,b_2,r\}$. (The freeness of the choice of $v_{ij}$ also implies that the vertices in $V_{ij}$ are twins.)
	
Suppose that $r$ is not adjacent to any vertices in $V(G) \setminus \{b_1,b_2,r\}$. Since $G$ is connected, we have $V_{ij} \neq V_{22}$. If $V_{ij} = V_{12}$, then $r$ is a leaf and $b_1$ is a cut-vertex. However, in~\cite{BasisForced} Corollary~7, it was shown that any finite graph has a metric basis that does not contain any cut-vertices. Thus, the cut-vertex $b_1$ cannot be a basis forced vertex, a contradiction. Similarly, if $V_{ij} = V_{21}$, then $b_2$ is a cut-vertex, a contradiction. Thus, suppose that $V_{ij} = V_{11}$. Now $d(r,v_{12})=2$, but $d(r,v_{22})>2$ which is a contradiction with \eqref{eqdist}. 
Therefore, the vertex $r$ is adjacent to all vertices in $V(G) \setminus \{b_1,b_2,r\}$ and the claim follows from the free choice of $v_{ij}$ (or the vertices in $V_{ij}$ being twins). Notice that now the vertices in $V_{ij}$ are true twins. ($\diamond$)
	
\medskip
\noindent	
\textbf{Claim 3:} $G \in \F$.

\smallskip
\noindent
\textit{Proof of Claim 3:} When $n=6$, the claim holds according to an exhaustive computer search.
	
Let $n \geq 7$. Now there exists a $V_{ij}$ such that $|V_{ij}|\geq 2$. Let $G'$ be the graph we obtain from $G$ by removing a vertex $w$ from a partition set $V_{ij}$ such that $|V_{ij}| \geq 2$ (recall that the vertices in $V_{ij}$ are true twins by Claim~2, so it does not matter which vertex we remove). The distances between vertices in $G'$ are exactly the same as the distances between the corresponding vertices in $G$, and thus if $R$ is a metric basis of $G$, then $R \cap V(G')$ is a resolving set of $G'$. Conversely, if $R'$ is a metric basis of $G'$, then $R' \cup \{w\}$ is a resolving set of~$G$. Therefore, $\dim (G') = \dim (G) - 1 = n(G') - 4$ and $b_1$ and $b_2$ are basis forced vertices also in~$G'$, as otherwise we would obtain a metric basis in $G$ such that $b_1$ or $b_2$ (or both) were not included in it. Let $G''$ be the graph we obtain from $G$ when we repeat the vertex removal until there is exactly one vertex left in each partition set. By the arguments above, it must happen that $n(G'')=6$, $\dim (G'') = n(G'') - 4 = 2$, and $G''$ has two basis forced vertices $b_1$ and $b_2$. Since $n(G'') = 6$, it must be that $G''$ is one of $G_1$, $G_2$ or $G_3$, and so, $G'' \in \F$. If $G'' = G_1$ or $G_2$, then clearly $G \in \F$, by the constructive arguments used.
	
Suppose that $G'' = G_3$. Now $b_1$ and $b_2$ are adjacent in~$G$ and the partition sets $V_{11}$ and $V_{22}$ in $G$ must have only one element each due to Claim~2. Let us then consider the graph $G^*$ that we obtain by repeating the same vertex removal process as above until we are left with a graph on seven vertices for which a next vertex removal would lead to the graph $G_3$. Now, without loss of generality, we can assume that the partition sets $V^*_{ij}$ in $G^*$ are such that $|V^*_{12}| = 2$ and $|V^*_{ij}| = 1$ otherwise. Thus, $G^*$ is as depicted in Figure~\ref{fig-Gstar}. By the arguments above, the graph $G^*$ should have two basis forced vertices. However, $G^*$ has two disjoint metric bases (which are illustrated in Figure~\ref{fig-Gstar}) and thus no basis forced vertices, a contradiction. Therefore, $G'' \neq  G_3$. ($\diamond$)

\medskip\noindent
To end the proof, we only need to remark that the arguments above allow to confirm that the two basis forced vertices $b_1$ and $b_2$ play the roles of $u$ and $v$ as defined in the graphs of the family $\F$.
\end{proof}

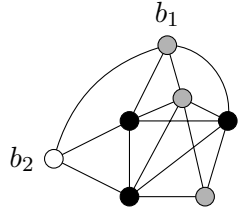
\begin{figure}
	\centering
	\begin{tikzpicture}[scale=1]
		\draw (0,0) -- (1,0) -- (.7,1.3) -- (1.3,1) -- (0,1) -- (.7,1.3) -- (0,0);
		\draw (.7,1.3) -- (.5,2) -- (0,1) -- (0,0) -- (-1,.5) -- (0,1)  (1,0) -- (1.3,1) -- (0,0);
		\draw (.5,2) to[out=0,in=80,looseness=1] (1.3,1);
		\draw[bend left] (-1,.5) to (.5,2);
		\draw \foreach \x in {(-1,.5)} {
			\x node[circle, draw, fill=white,
			inner sep=0pt, minimum width=7pt] {}
		};
		\draw \foreach \x in {(.5,2),(1,0),(.7,1.3)} {
			\x node[circle, draw, fill=gray!60,
			inner sep=0pt, minimum width=7pt] {}
		};
		\draw \foreach \x in {(0,1),(1.3,1),(0,0)} {
			\x node[circle, draw, fill=black,
			inner sep=0pt, minimum width=7pt] {}
		};
		\draw (.5,2) node[above, inner sep=7pt] {$b_1$};
		\draw (-1,.5) node[left, inner sep=7pt] {$b_2$};
	\end{tikzpicture}
	\caption{The graph $G^*$ with two disjoint metric bases illustrated as black and gray vertices.}\label{fig-Gstar} 
\end{figure}

\section{Concluding remarks}

We close our exposition with the following open questions that, in our opinion, might be of interest to continue this investigation.

\begin{itemize}
   \item It would be desirable to characterize (even at least partially) all the graphs $G$ for which $\bfv(G)\neq 0$. Specifically, the cases $\bfv(G)=1$ and $\bfv(G)=2$ would be of interest.
   \item In the light of Theorem~\ref{char2bf}, it would be interesting to characterize graphs with different metric dimensions and number of basis forced vertices.
   \item In view of Theorem \ref{th:main-bound}, if $\dim(G)=n-5$, then $G$ must also have at most two basis forced vertices. In this sense, a weaker version of the item above concerns finding the graphs $G$ of order $n$ with metric dimension $n-5$ and one or two basis forced vertices.
\end{itemize}

\section*{Acknowledgements}

Ville Junnila, Tero Laihonen, Anni Hakanen and Havu Miikonen have been partially supported by the Research Council of Finland grants 338797 and 358718. Ismael G. Yero has been partially supported by the Spanish Ministry of Science and Innovation through the grant PID2023-146643NB-I00.

\end{document}